\documentclass[11pt,a4paper]{article}
 
\usepackage[utf8]{inputenc}
\usepackage[english]{babel}
\usepackage[cal=esstix]{mathalfa}
\usepackage[T1]{fontenc}
\usepackage{amsmath}
\usepackage[hidelinks]{hyperref}
\usepackage{amsfonts}
\usepackage{amssymb}
\usepackage{amsthm}
\usepackage{tikz-cd}
\usepackage{mathtools}
\usepackage{multicol}
\usepackage{authblk}
\usepackage{changepage}
\usepackage[left=1.5cm,right=1.5cm,top=2.3cm,bottom=2.5cm, footskip=1.5cm]{geometry}
\usepackage{setspace}
\usepackage[style=alphabetic, backend=biber, doi=false, isbn=false, url=true]{biblatex}
\usepackage{fancyhdr}
\usepackage{url}

\AtEveryBibitem{
  \clearfield{note}
  \ifentrytype{book}{\clearfield{pages}}
}

\author{Gustave Billon}
\title{Branched Projective Structures and Bundles of Projective Frames on Surfaces}

\DeclareMathOperator{\Aff}{Aff}
\DeclareMathOperator{\paff}{\mathcal Aff}

\DeclareMathOperator{\PGL}{PGL}
\DeclareMathOperator{\id}{id}
\DeclareMathOperator{\ad}{ad}

\DeclareMathOperator{\At}{At}
\DeclareMathOperator{\CP}{\mathbb CP}
\DeclareMathOperator{\A}{A}
\DeclareMathOperator{\Hom}{Hom}

\DeclareMathOperator{\red}{red}

\DeclareMathOperator{\SO}{SO}
\DeclareMathOperator{\Aut}{Aut}
\DeclareMathOperator{\C}{\mathbb C}
\DeclareMathOperator{\R}{\mathbb R}

\DeclareMathOperator{\N}{\mathbb N}

\DeclareMathOperator{\Isom}{Isom}

\DeclareMathOperator{\PSL}{PSL}

\newtheorem{statement}{}[subsection]

\theoremstyle{definition}
\newtheorem{definition}[statement]{Definition}
\newtheorem*{definition*}{Definition}

\theoremstyle{theorem}
\newtheorem{theorem}[statement]{Theorem}
\newtheorem{proposition}[statement]{Proposition}
\newtheorem{corollary}[statement]{Corollary}
\newtheorem{lemma}[statement]{Lemma}

\theoremstyle{remark}
\newtheorem{remark}[statement]{Remark}

\begin{document}

\maketitle

\begin{abstract}
We show that the description of the holomorphic $\CP^1$-bundle associated to a holomorphic projective structure on a Riemann surface in terms of the principal bundle of projective $2$-frames extends very well to the setting of branched projective structures.
This generalization reveals a space of parameters, each of which is associated to a branching class.
The space of branched projective structures with a given branching class appears as a space of connections on a given $\CP^1$-bundle, and is consequently an affine space.
Finally, we study the map which to a branching class associates the corresponding $\CP^1$-bundle with section.
\end{abstract}

\section*{Introduction}

\subsection*{Projective Structures}

Holomorphic projective structures are $\left(\PSL(2, \C), \CP^1\right)$-structures, i.e. Riemann surfaces locally modeled on the Riemann sphere $\CP^1$.
More precisely, a holomorphic projective structure on a Riemann surface $X$ is given by an atlas $(U_i, f_i)_{i \in I}$, where $(U_i)_{i \in I}$ is an open cover of $X$ and $f_i : U_i \to \CP^1$ are a local biholomorphisms, such that on the intersections $U_i \cap U_j$, there is a Möbius transformation $g_{ij} \in \PSL(2, \C)$ such that $f_i = g_{ij} \circ f_j$.

These structures play a central role in the theory of uniformization of Riemann surfaces, which is the reason why they were introduced at the end of the nineteenth century (see \cite{saint-gervais}).
Recall that the uniformization theorem states that any simply connected Riemann surface is biholomorphic to $\CP^1$, $\C$ or $\mathbb H^2$, and can thus be seen as an open subset of the Riemann sphere.
In particular, if $X$ is a Riemann surface and $\pi : \widetilde X \to X$ its universal covering, then $\widetilde X \subset \CP^1$ and the group of deck transformations of $\pi$ is contained in $\PSL(2, \C)$, so that the local inverses of the covering map $\pi$ provide local identifications of $X$ with $\CP^1$, defining the so-called \emph{uniformizing projective structure} of $X$.

Projective structures also provide examples of \emph{opers} (see \cite{beilinson-drinfeld}) , namely $\PSL(2, \C)$-opers.
Recent papers extend properties of holomorphic projective structures to more general classes of opers (see in particular \cite{sanders}).
See also \cite{frenkel} for the role opers in the Langlands program.

This paper revisits one of the key properties of holomorphic projective structures.
Namely, if $X$ is a compact Riemann surface, the space of projective structures on $X$ is an affine space, directed by the vector space $H^0\left(X, K_X^{\otimes 2}\right)$ of global holomorphic quadratic differentials on $X$ (here $K_X$ denotes the cotangent bundle of $X$).

There is a very rich literature about this affine structure.
The classical approach is centered on an order $3$ differential operator, named the \emph{schwarzian derivative}~: the difference between two holomorphic projective structures $p_1$ and $p_2$ is obtained as the schwarzian derivative of the charts of $p_1$, seen as holomorphic functions of the coordinates given by the charts of $p_2$.
All the properties of the schwarzian derivative are encapsulated in the fact that this process defines a quadratic differential on $X$.
See for instance \cite{gunning}, \cite{dumas}, \cite{loray-marin} for a precise exposition of this approach.

\subsection*{The Projective Osculating Line}

A more geometric viewpoint on the affine structure of the space of projective structures on $X$ is developed in \cite{deligne_2}, \cite{kobayashi} and \cite{anderson}.
A holomorphic projective structure on $X$ is seen as a triple $(P, s, \nabla)$ where $P$ is a holomorphic $\CP^1$-bundle on $X$, $s$ is a holomorphic section of $P$ and $\nabla$ is a flat connection (equivalently a Riccati foliation) on $P$ such that $s$ and $\nabla$ are transverse over each point of $X$.
The key fact here is that the couple $(P, s)$ is the same for all the holomorphic projective structures on $X$.
As a consequence, the space of projective structures on $X$ appears as a space of connections on a fixed $\CP^1$-bundle, and inherits the classical affine structure of connection spaces.

We denote by $\left(P_X, s_X\right)$ the $\CP^1$-bundle with section associated to the holomorphic projective structures on $X$.
It is called the \emph{projective osculating line} of $X$ and should be understood as a projective equivalent to the tangent bundle $T_X$ of $X$.
The section $s_X$ plays the role of the zero section of $T_X$ : for any $x \in X$, $s_X(x)$ is the contact point between $X$ and the fiber $P_{X, x}$.
However, while $T_{X, x}$ has a tangency point of order $1$ with $X$, $P_{X, x}$ and $X$ are tangent up to order $2$ at $s(x)$.

In \cite{kobayashi} and \cite{anderson}, the projective osculating line of $X$ is introduced by way of its principal bundle of trivializations, namely the bundle of \emph{projective $2$-frames} on $X$.
This is also the approach of this paper.

\subsection*{Branched Projective Structures}

The aim of this article is to generalize this point of view on projective structures, involving the projective osculating line, to \emph{branched projective structures}.
A branched projective structure is a projective structure with a particular kind of singularities, namely cone points with angle a multiple of $2\pi$.
Formally, a branched projective structure on the Riemann surface $X$ is given by an atlas $(U_i, f_i)$, still with Möbius transformations as changes of charts, but where the charts $f_i$ are now only nonconstant holomorphic maps.
Here the differentials $df_i$ may vanish, and the vanishing locus of the $df_i$ define a divisor $D$ on $X$, named the \emph{branching divisor} of the branched projective structure.

Branched projective structures arise for instance as pullbacks of projective structures by ramified coverings.
It should be noted that this notion is a special case (where poles are of order at most $2$ and local monodromy is trivial) of the more general notion of \emph{meromorphic projective structure}, that is studied in particular in \cite{allegretti-bridgeland}, \cite{gupta-mj}, \cite{gupta-mj_2} and \cite{serandour}.
However, usual techniques for the study of meromorphic projective structures do not apply to branched projective structures.

Branched projective structures are also examples of \emph{branched opers}, namely branched $\PSL(2, \C)$-opers.
See \cite{frenkel-gaitsgory} and \cite{biswas-dumitrescu-heller} for more on branched opers.
It should also be noted that branched projective structures have been extensively studied with regards to their holonomy representations, that are elements of $\Hom\left(X, \PSL(2, \C)\right)$, see \cite{gallo-kapovich-marden}, \cite{calsamiglia-deroin-francaviglia}, \cite{le_fils}.
While the holonomy representations of unbranched projective structures are all nonelementary, any element in $\Hom\left(X, \PSL(2, \C)\right)$ is the holonomy of a branched projective structure.
This illustrates the flexibility allowed by branching singularities.

One difficulty in the study of branched projective structures is that, unlike what happens for (meromorphic) projective structures, the space of branched projective structures over a fixed Riemann surface $X$ and divisor $D$ is not an affine space in general.
It is only an analytic space, as shown in \cite{mandelbaum_1} where Mandelbaum introduced branched projective structures.

One way to understand this lack of structure is to consider the description of branched projective structures as flat $\CP^1$-bundles  with section.
Similarly to the unbranched case, a branched projective structure on $X$ is a triple $(P, s, \nabla)$ where $P$ is a holomorphic $\CP^1$-bundle, $s$ a holomorphic section of $P$ and $\nabla$ a holomorphic connection on $P$ such that $s$ is not flat for $\nabla$ (equivalently, $s$ is not a leaf of the Riccati foliation associated to $\nabla$).
Note that, unlike in the case of projective structures, here the condition on the triple $(P, s, \nabla)$ is a generic one, in this sense the notion of branched projective structures is a generic generalization of the concept of unbranched projective structures (like branched coverings are generic generalizations of coverings).

This description of branched projective structures reveals a major difference with the unbranched case.
Namely, the couple $(P, s)$ depends on the branched projective structure considered, even when the underlying Riemann surface $X$ and branching divisor $D$ are fixed.
This explains why the space of branched projective structures over a Riemann surface with divisor is not an affine space, being not a space of connections.
Note that in \cite{biswas-dumitrescu-gupta} the space of branched projective structures over a curve with divisor is described as a space of \emph{logarithmic} connections over a fixed rank $2$ vector bundle, while in \cite{biswas-dumitrescu}, generic triples $(P, s, \nabla)$ are described as \emph{branched $\SO(3, \C)$-opers}.

In this paper we show that the construction of the projective osculating line in the unbranched case naturally generalizes to the branched case and exhibits, for a fixed $X$ and $D$, a space of parameters that we call \emph{branching classes}.
To each branching class is associated a \emph{branched projective osculating line}, that is a $\CP^1$-bundle with section, generalizing the twisted tangent line bundle $T_X(D)$ with the zero section.
To a branched projective structure is associated a branching class, and the $\CP^1$-bundle with section associated to a branched projective structure is the projective osculating line corresponding to its branching class.
One consequence is that the space of branched projective structures corresponding to a given branched osculating line, being identified to a space of connections, is an affine space.
The underlying vector space is $H^0\left(X, K_X^{\otimes 2}(-D)\right)$.

Note that in \cite{billon}, branching classes is the key element to construct an analytic structure on the space of all branched projective structures with fixed genus and branching degree, and investigate their singularities.
See also \cite{billon_2}.

It should be noted that a branched projective osculating line is slightly more that a $\CP^1$-bundle with section. In particular several branched projective osculating lines may have the same underlying $\CP^1$-bundle with section.
We conclude this paper by studying the map from the space of branching classes (or branched projective osculating lines), that we show to be an affine space, to the moduli space of $\CP^1$-bundles with section, in the spirit of \cite{maruyama}.

\subsection*{Structure of the Paper}

In section \ref{section_cartan_geometries}, we present the theory of $(G, X)$-structures and Cartan geometries with a point of view that will be useful to consider projective structures from the viewpoint of the projective osculating line.
We emphasize the equivalence between the datum of a bundle (with section) over a manifold and the datum of the principal bundle of its trivializations (preserving the section).

In section \ref{section_projective_structures_projective_osculating_line}, we introduce the theory of holomorphic projective structures on a fixed Riemann surface $X$, in the spirit of \cite{kobayashi-nomizu}, \cite{kobayashi} and \cite{anderson}, through the study of the osculating projective line $P_X$, its adjoint bundle and its Atiyah bundle.
With our viewpoint, the object that arises naturally is in fact the bundle of  projective $2$-frames of $X$, that is the bundle of trivializations of $P_X$.
Thus we mostly work with projective $2$-frames.
Projective structures arise as some class of Cartan connections on the bundle of projective $2$-frames, called \emph{projective connections}.
We show that these connections induce a natural structure of $\SO(3, \C)$-oper on the adjoint bundle of the bundle of projective $2$-frames.
The content of this section, as well as the previous one, is classical.

In section \ref{section_bps_bol}, we confront the viewpoint developed in section \ref{section_projective_structures_projective_osculating_line} on the projective osculating line to the branched case.
This reveals, at each branching point, a parameter space, the space of branching classes, each branching class corresponding to a branched analog of the bundle of projective $2$-frames, thus leading to a branched analog of the osculating projective line.
We then generalize the constructions carried out for projective $2$-frames to branched projective $2$-frames.
A notion of \emph{branched projective connections} on these bundles arises, that we show to be equivalent to the notion of branched projective structures.

In section \ref{section_spaces_bpf}, we are interested in the space of parameters defined in section \ref{section_bps_bol}.
We show that this space of parameters is an affine space, and we study the map from this affine space to the moduli space of $\CP^1$-bundles with section, that maps each branching class to the $\CP^1$-bundle with section given by the associated branched osculating line.

\subsection*{Acknowledgements}

I am thankful to my advisor Sorin Dumitrescu for his help and support.
I would like to thank David Dumas for sharing the work of Charles Gregory Anderson \cite{anderson} with me.

\section{Cartan Geometries}\label{section_cartan_geometries}

In this section we introduce the language of Cartan geometries, that we will use to describe the $\CP^1$-bundle associated to a projective structure.

\subsection{Different Notions of Manifolds Modeled on a Space}\label{subsection_diff_notions}

Let $Q$ be a manifold, endowed with the faithful left-action of a Lie group $G$ by diffeomorphisms.
For instance, one can have in mind the space $\R^n$, along with the action of its group of isometries $\Isom(\R^n)$.
In this article we are interested in the projective line $\CP^1$ with the action of the projective linear group $\PGL(2, \C)$ by linear transformations.
The action of $G$ makes it possible to pay attention to properties of some objects of $Q$, that we call \emph{geometric properties}.
A property $p$ is said to be geometric if, for any object $\omega$ attached to $Q$ that has the property $p$, the image of $\omega$ by any element of $G$ also has the property $p$.
For instance, in $\R^n$ with the action of $\Isom(\R^n)$, the property "being aligned" is a geometric property of sets of points.
The property "being of norm $1$" is a geometric property of tangent vectors.
In $\CP^1$ with the action of $\PGL(2, \C)$, four points can have the property "being of cross-ratio $1$", which is a geometric property.

Let us first introduce an abstract way of considering the geometry of $Q$ with the action of $G$ :

\begin{definition}
A $(G, Q)$-\emph{space} is a space $Q'$ endowed with a family $(\alpha_i)_{i \in I}$ of \emph{trivializations}, that is to say of diffeomorphisms $\alpha_i : Q' \xrightarrow{\sim} Q$, such that for any $i, j \in I$ there exists $g \in G$ such that $\alpha_i = g \circ \alpha_j$.
If $Q_1$ and $Q_2$ are $(G, Q)$-spaces, a diffeomorphism $\phi : Q_1 \xrightarrow{\sim} Q_2$ is said to be an \emph{isomorphism} of $(G, Q)$-spaces if there is a trivialization $\alpha_1 : Q_1 \xrightarrow{\sim} Q$ of $Q_1$ and a trivialization $\alpha_2 : Q_2 \xrightarrow{\sim} Q$ of $Q_2$ such that $\alpha_1 = \alpha_2 \circ \phi$
\end{definition}

\begin{remark}
\begin{itemize}
\item[(i)] A $(\Isom(\R^n), \R^n)$-space is a Euclidean space of dimension $n$.
\item[(ii)] If $Q'$ is a $(G, Q)$-space, there is no canonical action of $G$ on $Q'$.
\item[(iii)] Since the action of $G$ on $Q$ is faithful, the set of trivializations of a $(G, Q)$-space $Q'$ is a $G$-\emph{torsor}, i.e. a space on which $G$ acts freely and transitively.
In particular it inherits the differential structure of $G$.
\end{itemize}
\end{remark}

If the geometric properties that can be studied on $Q$ are of interest, one might ask whether the same notions can also be studied on objects attached to other differential manifolds, that are not necessarily isomorphic to $Q$.
Let $M$ be a differential manifold.
If one is interested in local objects, it is enough to only have, for each point $m \in M$, an identification of an open neighborhood of $m$ with an open subset of $Q$. 
Of course these several local identifications have to be consistent with the geometry of $Q$~: if two of them are defined on the same open subset of $M$, one has to be obtained from the other by composing with the action of an element of $G$, so that the induced local geometry on $M$ is the same.
This is the notion of $(G, Q)$-structure~:

\begin{definition}
Let $M$ be a differential manifold.
A $(G, Q)$-atlas on $M$ is the datum of~:
\begin{itemize}
\item An open cover $(U_i)_{i \in I}$ of $M$
\item For each $i \in I$, a map $f_i : U_i \to Q$ that is a diffeomorphism from $U_i$ to its image, such that for any $i, j \in I$ there exists $g \in G$ with $f_i(x) = g \cdot f_j(x)$ for all $x \in U_i \cap U_j$
\end{itemize}
Two $(G, Q)$-atlases are equivalent if their union is still a $(G, Q)$-atlas.
A $(G, Q)$-structure is an equivalence class of $(G, Q)$-atlases.
\end{definition}

For example, a $(\Isom(\R^n), \R^n)$-structure is a flat Riemannian manifold.

Another way of understanding the fact that $M$ locally has the geometry of $Q$ is the following.
Consider a small observer moving in $Q$.
In the case where $M$ is a sphere, $Q = \R^2$ and $G = \Isom(\R^2)$, one can think of somebody walking on a very big sphere, such as the earth.
Wherever he is, the observer sees his neighborhood as if he was living in $Q$~: to each point $m$ of $M$ is attached a copy $\mathcal Q_m$ of $Q$, and the observer locates himself as if he was moving in $\mathcal Q_m$. 
In the case of the plane $\R^2$ with the action of $\Isom(\R^2)$, the space $\mathcal Q_m$ would be the tangent space of $M$ at $m$. 
In particular, there is a point $s(m) \in \mathcal Q_m$, identified with $m \in M$ that is the point where the observer considers he is standing when he is in $m$. 
When the observer moves from a point $m \in M$ to $m' \in M$, he feels like he was moving inside $Q$, while he moves from $\mathcal Q_m$ to $\mathcal Q_{m'}$~: he identifies these two spaces.
Moreover, when he moves inside $M$, the observer sees that his position inside his reference space changes, meaning that $s_{m'}$ is \emph{not} identified to $s_m$ when $m$ and $m'$ are close. 
In other words, the family of spaces $(\mathcal Q_m)_{m \in M}$ comes with a \emph{connection}, i.e. an identification between $\mathcal Q_{\gamma(0)}$ and $\mathcal Q_{\gamma(1)}$ for any smooth path $\gamma : [0, 1] \to M$. 
This connection is \emph{transverse} to the family of points $(s_m)_{m \in M}$. 
Such a local modelling of $M$ on the geometry of $Q$ is formalized by the notion of \emph{Cartan geometry}.

\begin{definition}\label{def_cartan_geom}
A \emph{Cartan geometry} on $M$ modeled on $Q$ with the action of $G$ is given by
\begin{itemize}
\item A differential bundle of $(G, Q)$-spaces $\pi : \mathcal Q \to M$ over $M$
\item A connection $\nabla$ on $\mathcal Q$, i.e. a distribution of $\dim(M)$-dimensional spaces in $\mathcal Q$, transverse to the fibers of $\pi$, and such that the associated parallel transport identifies the fibers by isomorphisms of $(G, Q)$-spaces
\item A section $s$ of $\mathcal Q$ that is transverse to the connection $\nabla$.
\end{itemize}
\end{definition}

For instance, a Cartan geometry modeled on $\R^n$ with the action of $\Isom(\R^n)$ is a (not necessarily flat) Riemannian manifold.

\begin{remark}
\begin{itemize}
\item[(i)]We use in definition \ref{def_cartan_geom} the notion of \emph{differential bundle of $(G, Q)$-spaces}.
Let us give a formal definition, although it is not surprising.
A differential bundle of $(G, Q)$-spaces on $M$ is a differential manifold $\mathcal Q$ with a differential map $\pi : \mathcal Q \to M$ as well as an open covering $(U_i)_{i \in I}$ of $M$ and for each $i \in I$ a \emph{local trivialization}, i.e. a diffeomorphism $(\pi, F_i) : \pi^{-1}(U_i) \xrightarrow{\sim} U_i \times Q$, such that the restrictions $f_{i, m} = F_i|_{\pi^{-1}(m)} : \pi^{-1}(m) \to Q$ ($m \in M$) are isomorphisms of $(G, Q)$-spaces, and for any $i, j \in I, m \in M$, there exists $g_{ij} \in G$ such that $f_{i, m} = g_{ij} \circ f_{j, m}$.
\item[(ii)] A bit trickier is the notion of \emph{section} of a bundle of $(G, Q)$-spaces.
A section of the $(G, Q)$-bundle $\mathcal Q$ is a smooth map $s : M \to \mathcal Q$ such that, firstly, $\pi \circ s = \id_M$ and, secondly, there is an open cover $(U_i)_{i \in I}$ of $M$ and trivializations $(\pi, F_i) : \pi^{-1}(U_i) \xrightarrow{\sim} U_i \times Q$ such that for each $i \in I$, $F_i \circ s: U_i \to Q$ is constant.
The second condition is a consequence of the first one only in the case where the action of $G$ on $Q$ is transitive.
\item[(iii)] The distribution defined by the connection $\nabla$ has rank $\dim(M)$. 
The section $s$, seen as a submanifold of $\mathcal Q$, also has dimension $\dim(M)$. 
Thus the transversality of $s$ and $\nabla$ gives $2 \dim(M) = \dim(\mathcal Q) = \dim(M) + \dim(Q)$, so $\dim(M) = \dim(Q)$.
\end{itemize}
\end{remark}

Suppose from now on that the action of $G$ on $Q$ is \emph{analytic}, meaning that for any nonempty open set $U \subset Q$, an element $g \in G$ is uniquely determined by the restriction of its action to $U$.
Recall that a connection is said to be \emph{flat} if the associated distribution is integrable.

Suppose given a $(G, Q)$-structure on $M$, with atlas $(U_i, f_i)_{i \in I}$.
For each $i \in I$, consider a $(G, Q)$-space $Q_i$, as well as an open subset $V_i \subset Q_i$ and an identification $\iota_i : V_i \xrightarrow{\sim} U_i$ (the diffeomorphism $f_i$ ensures the existence of $V_i$ and $\iota_i$).
Since the action of $G$ is analytic, for each $i \in I$ there is a unique trivialization $\gamma_i : Q_i \xrightarrow{\sim} Q$ extending the map $f_i \circ \iota_i$.
Thus for each nonempty intersection $U_i \cap U_j$, there is a uniquely determined identification $\gamma_i^{-1} \circ \gamma_j : Q_j \to Q_i$.
These identifications make the family $(Q_i)_{i \in I}$ into a \emph{local system} of $(G, Q)$-spaces on $Q$.
Such a local system can be seen as a flat bundle of $(G, Q)$-spaces $\mathcal Q$ over $M$, by considering for each $i \in I$ the product $U_i \times Q_i$ and glueing these products over the intersections $U_i \cap U_j$ by the isomorphisms $(m, x_j) \in (U_i \cap U_j) \times Q_j \mapsto \left(m, \gamma_i^{-1} \circ \gamma_j(x_j)\right) \in (U_i \cap U_j) \times Q_i$.
The connection on $\mathcal Q$ is given locally over $U_i$ by the horizontal distribution on the trivial bundle $U_i \times Q_i$.
Moreover, the identifications $\iota_i : V_i \xrightarrow{\sim} U_i$ are encoded in a section of $\mathcal Q$ given on each $U_i$ by the diagonal section $m \in U_i \mapsto \iota_i^{-1}(m) \in U_i \times Q_i$, which is transverse to the flat connection.
In other words, a $(G, Q)$-structure on $M$ induces a flat Cartan geometry on $M$, i.e. a Cartan geometry whose associated connection is flat.

Reciprocally, a flat Cartan geometry $(\mathcal Q, \nabla, s)$ on $M$ induces a $(G, Q)$-structure on $M$~: the flat bundle $(\mathcal Q, \nabla)$ defines a local system $(U_i, Q_i)_{i \in I}$ of $(G, Q)$-spaces on $Q$, and the section $s$ gives identifications of each open subset $U_i \subset M$ with an open subset of $Q_i$.

\begin{proposition}
The datum of a $(G, Q)$-structure on $M$ is equivalent to the datum of a flat Cartan geometry modeled of $Q$ with the action of $G$.
\end{proposition}

\subsection{Adjoint and Atiyah Bundles}\label{subsection_adjoint_atiyah}

From now on, suppose that the action of $G$ on $Q$ is transitive, so that all points of $Q$ are geometrically equivalent.
Fix a point $x_0 \in Q$, so that $Q$ is now a \emph{pointed} space.
Let $H \subset G$ be the subgroup of $G$ fixing $x_0$, so that $Q = G/H$, and $\left(TQ\right)_{x_0} = \mathfrak g / \mathfrak h$, where $\mathfrak g$ and $\mathfrak h$ are the respective Lie algebras of $G$ and $H$.
One can slightly modify the definition of a $(G, Q)$-space to get the definition of a $(G, Q, x_0)$-space~:

\begin{definition}
A $(G, Q, x_0)$-\emph{space} is the datum of a $(G, Q)$-space $Q'$ with a marked point $x_0' \in Q'$.
The \emph{trivializations} of the $(G, Q, x_0)$-\emph{space} $(Q', x_0')$ are the trivializations of $Q'$ sending $x_0'$ to $x_0 \in Q$.
\end{definition}

There is an obvious notion of \emph{isomorphism of $(G, Q, x_0)$-spaces}.
There is also a notion of \emph{bundle of $(G, Q, x_0)$-spaces}, that is equivalent to the notion of bundle of $(G, Q)$-spaces with section.
In particular, the structure group of a bundle of $(G, Q, x_0)$-spaces is $H$.
There is also a notion of connection on a bundle of $(G, Q, x_0)$-spaces, for which the parallel transport has to preserve the section.

Take a $(G, Q)$-space (respectively a $(G, Q, x_0)$-space) $Q_1$.
The \emph{adjoint space} of $Q_1$, denoted by $\ad(Q_1)$, is the vector space of infinitesimal automorphisms of $Q_1$, i.e. the space of global vector fields of $Q_1$ whose flow at any small time is an automorphism of $Q_1$.
The vector space $\ad(Q_1)$ is stable for the Lie bracket of vector fields, and is thus a Lie algebra.
A trivialization of $Q$ identifies the Lie algebra $\ad(Q_1)$ with $\mathfrak g$ (respectively identifies the Lie algebra $\ad(Q_1)$ with $\mathfrak h$, and the vector space $T_{x_0}Q_1$ with $\mathfrak g / \mathfrak h$).
If $x_1 \in Q_1$, and $\widetilde Q_1$ is the $(G, Q, x_0)$-space given by $Q_1$ and the marked point $x_1$, the infinitesimal automorphisms of $\widetilde{Q_1}$ are in particular infinitesimal automorphisms of $Q_1$, thus $\ad\left(\widetilde{Q_1}\right) \subset \ad(Q_1)$.
The tangent space of $Q_1$ at $x_1$ is given by $\left(T_{Q_1}\right)_{x_1} = \ad\left(Q_1\right)/\ad\left(\widetilde{Q_1}\right)$.

Now take a bundle of $(G, Q)$-spaces (respectively of $(G, Q, x_0)$-spaces) $\pi : \mathcal Q \to M$.
The \emph{adjoint bundle} of $\mathcal Q$, denoted by $\ad(\mathcal Q)$ is the vector bundle on $M$ whose fiber over $m \in M$ is $\ad(\mathcal Q_m)$.
It is a bundle of Lie algebras.

The \emph{Atiyah bundle} of $\mathcal Q$, denoted by $\At(\mathcal Q)$, is the vector bundle over $M$ whose fiber over $m \in M$ is the vector subspace of $H^0\left(\pi^{-1}(m), T_{\mathcal Q}|_{\pi^{-1}(m)}\right)$ generated by global sections $\mathcal V$ of $T_{\mathcal Q}|_{\pi^{-1}(m)}$ for which there is a vector $V \in \left(T_M\right)_m$ and a local trivialization of $\mathcal Q$ such that $\mathcal V$ is the horizontal lift of $V$.
See \cite{atiyah}.

There is a natural notion of isomorphism of bundles of $(G, Q)$-spaces (respectively $(G, Q, x_0)$-spaces).
Given another bundle of $(G, Q)$-spaces (respectively $(G, Q, x_0)$-spaces) $\pi' : \mathcal Q' \to M'$, where $M'$ is another differential manifold, a map $\phi : \mathcal Q \to \mathcal Q'$ is an isomorphism if $\phi$ is a diffeomorphism, there exists an underlying diffeomorphism $\psi : M \xrightarrow{\sim} M'$ such that  $\pi' \circ \phi = \psi \circ \pi$, and $\phi$ induces isomorphisms of $(G, Q)$-spaces (respectively $(G, Q, x_0)$-spaces) between fibers of $\pi$ and fibers of $\pi'$.
A \emph{gauge transformation} of $\mathcal Q$ is an automorphism of $\mathcal Q$ for which the underlying automorphism of $M$ is the identity.
The adjoint bundle of $\mathcal Q$ can be seen as the bundle of infinitesimal gauge transformations of $\mathcal Q$.
The Atiyah bundle can be seen as the bundle of infinitesimal automorphisms of $\mathcal Q$.
In particular, $\ad(\mathcal Q) \subset \At(\mathcal Q)$.

The exact sequence at any $x \in \mathcal Q$
\begin{equation}\label{eq_tangent_of_bundle_exact}
0 \to \ker (d\pi)_x \to \left(T_{\mathcal Q}\right)_x \xrightarrow{d\pi_x} \left(T_M\right)_{\pi(x)} \to 0
\end{equation}
defines an exact sequence
\begin{equation}\label{eq_ad_at_exact}
0 \to \ad(\mathcal Q) \to \At(\mathcal Q) \to T_M \to 0
\end{equation}
That is called the \emph{Atiyah exact sequence} associated to $\mathcal Q$.

Since the infinitesimal automorphisms of $Q$ are identified with elements of the Lie algebra $\mathfrak g$ of $G$ (respectively the lie algebra $\mathfrak h$ of $H$), $\ad(\mathcal Q)$ is of rank $\dim(G)$ (respectively $\dim(H)$).
As a consequence, it follows from the Atiyah exact sequence \eqref{eq_ad_at_exact} that $\At(\mathcal Q)$ is of rank $\dim(G) + \dim(M)$ (respectively $\dim(H) + \dim(M)$).

A connection on $\mathcal Q$, seen as a distribution on $\mathcal Q$ transverse to the fibers of $\pi : \mathcal Q \to M$, defines at each point $x \in \mathcal Q$ a decomposition $\left(T_{\mathcal Q}\right)_x = \left(T_M\right)_{\pi(x)} \oplus \ker (d\pi)_x$, where $\left(T_M\right)_{\pi(x)}$ is identified with the vector subspace of $\left(T_{\mathcal Q}\right)_x$ given by the connection at $x$.
In particular, it defines a splitting of the exact sequence \eqref{eq_tangent_of_bundle_exact}, which in turn induces a splitting of the exact sequence \eqref{eq_ad_at_exact}.
In fact, the datum of a connection on $\mathcal Q$ is equivalent to the datum of a splitting of \eqref{eq_ad_at_exact}~: such a splitting enables to lift a vector field on $M$ to a vector field on $\mathcal Q$ whose restriction to each fiber is in the Atiyah bundle, which ensures that the associated parallel transport identifies the fibers of $\pi$ through isomorphisms of $(G, Q)$-spaces (respectively $(G, Q, x_0)$-spaces).

\subsection{Cartan Connections}\label{subsection_cartan_connections}

Consider on $M$ a Cartan geometry modeled on $(Q, G)$, given by a triple $(\mathcal Q, \nabla, s)$, where $\pi : \mathcal Q \to M$ is a bundle of $(G, Q)$-spaces, $\nabla$ is a connection on $\mathcal Q$ and $s$ is a section transverse to $\nabla$.
In particular the couple $(\mathcal Q, s)$ can be seen as a $(G, Q, x_0)$-bundle, that we denote by $\widetilde{\pi} : \widetilde{\mathcal Q} \to M$.
Note that $\widetilde{\mathcal Q}$ is a reduction of $\mathcal Q$ to the structure group $H$.
One has the inclusion $\ad\left(\widetilde{\mathcal Q}\right) \subset \ad(\mathcal Q)$.
The local trivializations of $\widetilde{\mathcal Q}$ are also local trivializations of $\mathcal Q$, which on an infinitesimal level implies $\At\left(\widetilde{\mathcal Q}\right) \subset \At(\mathcal Q)$.
If $m \in M$, denote by $Q_m$ (respectively $\widetilde Q_m$) the fiber of $\mathcal Q$ (respectively $\widetilde{\mathcal Q}$) over $m$.
One has $\left(T_{Q_m}\right)_{s(m)} = \ad(\mathcal Q)_m / \ad \left(\widetilde{\mathcal Q}\right)_m$.

Let $\left(W_{\xi}\right)_{\xi \in \mathcal Q}$ be the distribution on $\mathcal Q$ defining $\nabla$, in particular $W_{\xi} \subset \left(T_{\mathcal Q}\right)_{\xi}$ and $\left(T_{\mathcal Q}\right)_{\xi} = W_{\xi} \oplus \ker(d\pi)_{\xi}$.
Let also $\omega : \At(\mathcal Q) \to \ad(\mathcal Q)$ be the splitting of the exact sequence $0 \to \ad(\mathcal Q) \to \At(\mathcal Q) \to T_M \to 0$ given by $\nabla$. If $m \in M$ and $\mathcal V \in \At(\mathcal Q)_m \subset H^0\left(Q_m, T_{\mathcal Q}|_{Q_m}\right)$, then $\omega(\mathcal V) = 0$ if and only if $\mathcal V_{\xi} \in W_{\xi}$ for all $\xi \in Q_m$.

Since $\omega$ is a splitting, it is the identity on $\ad(\mathcal Q)$ and thus on $\ad\left(\widetilde{\mathcal Q}\right)$.
Suppose now $\mathcal V \in \At \left(\widetilde{\mathcal Q}\right)_m \subset \At(\mathcal Q)_m$.
In particular $\mathcal V_{s(m)}$ is tangent to the section $s$.
Since the connection $\nabla$ is transverse to $s$, $\omega(\mathcal V) = 0$, i.e. $\mathcal V_{s(m)} \in W_{s(m)}$, if and only if $\mathcal V_{s(m)} = 0$, which implies $d\pi(\mathcal V) = 0$, so $\mathcal V \in \ad \left(\widetilde{\mathcal Q}\right)$ and thus $\mathcal V = 0$ because $\omega$ is the identity on $\ad \left(\widetilde{\mathcal Q}\right)$.
As a consequence, $\omega$ is injective when restricted to $\At \left(\widetilde{\mathcal Q}\right)$.
Recall that $\ad(\mathcal Q)$ is of rank $\dim(G)$ and $\At \left(\widetilde{\mathcal Q}\right)$ is of rank $\dim(H) + \dim(M)$.
Since $\dim(M) = \dim(Q) = \dim(G)-\dim(H)$, the vector bundles $\ad(\mathcal Q)$ and $\At \left(\widetilde{\mathcal Q}\right)$ have the same rank.
So the injective morphism $\omega|_{\At \left(\widetilde{\mathcal Q}\right)}$ is an isomorphism.
Reciprocally, a splitting $\omega' : \At(\mathcal Q) \to \ad(\mathcal Q)$ defines a connection on $\mathcal Q$, and if moreover $\omega'|_{\At \left(\widetilde{\mathcal Q}\right)}$ is an isomorphism (in particular injective), the induced connection is transverse to $s$.
Notice that the vector subbundles $\At \left(\widetilde{\mathcal Q}\right)$ and $\ad(\mathcal Q)$ generate the vector bundle $\At(\mathcal Q)$, so that such a splitting $\omega'$ is determined by its restriction to $\At \left(\widetilde{\mathcal Q}\right)$.

\begin{definition}
A \emph{Cartan connection} on a bundle of $(G, Q)$-spaces $\mathcal Q$ on $M$ along with a reduction to a bundle of $(G, Q, x_0)$-spaces $\widetilde{\mathcal Q}$ is an isomorphism $\omega|_{\At \left(\widetilde{\mathcal Q}\right)} : \At \left(\widetilde{\mathcal Q}\right) \xrightarrow{\sim} \ad(\mathcal Q)$ that is the identity in restriction to $\ad \left(\widetilde{\mathcal Q}\right)$.
\end{definition}

\begin{proposition}
The datum of a Cartan geometry on $M$ modeled on $Q$ with the action of $G$ is equivalent to the datum of a bundle of $(G, Q)$-spaces $\mathcal Q$ on $M$, along with a reduction to a bundle of $(G, Q, x_0)$-spaces $\widetilde{\mathcal Q}$, and a Cartan connection on $\mathcal Q$ along with its reduction $\widetilde{\mathcal Q}$.
\end{proposition}

Consider a Cartan geometry on $M$ given by a bundle of $(G, Q)$-spaces $\mathcal Q$, a bundle of $(G, Q, x_0)$-spaces $\widetilde{\mathcal Q}$ and an isomorphism $\omega|_{\At \left(\widetilde{\mathcal Q}\right)} : \At \left(\widetilde{\mathcal Q}\right) \xrightarrow{\sim} \ad(\mathcal Q)$ that is the identity on $\ad \left(\widetilde{\mathcal Q}\right)$.
The isomorphism $\omega|_{\At \left(\widetilde{\mathcal Q}\right)}$ defines for each $m \in M$ an isomorphism of exact sequences~:
\begin{equation}\label{eq_iso_exact_cartan_connection}
\begin{tikzcd}
0 \arrow[r] &\ad \left(\widetilde{\mathcal Q}\right)_m \arrow[d, equal] \arrow[r] & \At \left(\widetilde{\mathcal Q}\right)_m \arrow[d, "\omega|_{\At \left(\widetilde{\mathcal Q}\right)_m}"] \arrow[r] & \left(T_M\right)_m \arrow[d, "\Phi_m"] \arrow[r] & 0 \\
0 \arrow[r] & \ad \left(\widetilde{\mathcal Q}\right)_m \arrow [r] & \ad(\mathcal Q)_m \arrow[r] & \left( T_{Q_m} \right)_{s(m)} \arrow[r] & 0
\end{tikzcd}
\end{equation}
In particular, the Cartan geometry comes with an identification of the tangent space of $M$ at each $m$ with the tangent space of $Q_m$ at $s(m)$.
This confirms that the space $Q_m$ is an osculating space to $Q$ at $m$, the point $s(m)$ being the point of contact.

\subsection{The Point of View of Principal Bundles}\label{subsection_principal_bundles}

Consider a $(G, Q)$-space $Q'$, and let $\mathcal T_{Q'}$ be the space of trivializations of $Q'$.
The group $G$ acts freely and transitively on the left on $\mathcal T_{Q'}$, so that $\mathcal T_{Q'}$ is a \emph{$G$-torsor}.
A trivialization $\alpha_i \in \mathcal T_{Q'}$ identifies the group $G$ to $\mathcal T_{Q'}$ by $g \in G \mapsto g \circ \alpha_i \in \mathcal T_{Q'}$.
The differential structure induced on $\mathcal T_{Q'}$ by this identification does not depend on the choice of $\alpha_i$.
Note that there is a surjective map
\begin{equation}
\begin{array}{rrcl}
\Phi :& \mathcal T_{Q'} \times Q & \to & Q' \\
&(\alpha, x) & \mapsto & \alpha^{-1}(x)
\end{array}
\end{equation}
and the fibers of $\Phi$ are the orbits of the left-action of $G$ on $\mathcal T_{Q'} \times Q$.
Thus $\Phi$ defines a canonical isomorphism $G \backslash \left(\mathcal T_{Q'} \times Q \right) \xrightarrow{\sim} Q'$.

Reciprocally, given a $G$-torsor $\mathcal T$, one can consider the quotient 
\begin{equation}\label{eq_bundle_as_quotient_principal}
Q' = G \backslash \left(\mathcal T \times Q \right)
\end{equation}
Since the action of $G$ on $\mathcal T$ is free and transitive, if $t_0 \in \mathcal T$, any $x' \in Q'$ has a unique representative in $\mathcal T \times Q$ of the form $(t_0, x)$.
Thus $t_0$ defines a diffeomorphism
\begin{equation}
\begin{array}{rrcl}
\alpha_{t_0} :& Q' & \to & Q \\
&[(t_0, x)] & \mapsto & x
\end{array}
\end{equation}

\begin{remark}\label{rk_quotient_principal_by_isotropies}
Having in mind the preferred point $x_0 \in Q$, fixed by the subgroup $H \subset G$, one has $Q' = H \backslash \mathcal T$, where if $t \in \mathcal T$, the class $[t] \in H \backslash \mathcal T$ is identified with the class $[t, x_0] \in G \backslash \left(\mathcal T \times Q \right) = Q'$.
\end{remark}

If $g \in G$, one has $\alpha_{g \cdot t_0} = g \circ \alpha_{t_0}$.
This proves that the diffeomorphisms $\left(\alpha_t\right)_{t \in \mathcal T}$ are the trivializations for a structure of $(G, Q)$-space on $Q'$.
In particular, the $G$-torsor $\mathcal T$ is canonically identified through the map $\alpha$ with the $G$-torsor $\mathcal T_{Q'}$ of trivializations of $Q'$.
The same holds, \emph{mutatis mutandis}, in the case of $(G, Q, x_0)$-spaces.
We have proved the following proposition~:

\begin{proposition}
The datum of a $(G, Q)$-space $Q'$ is equivalent to the datum of the $G$-torsor $\mathcal T_{Q'}$ of its trivializations.
The datum of a $(G, Q, x_0)$-space is equivalent to the datum of the $H$-torsor of its trivializations.
\end{proposition}

Fix a $(G, Q)$-space (respectively a $(G, Q, x_0)$-space) $Q'$ and its torsor of trivializations $\mathcal T_{Q'}$.
The group $\Aut(Q')$ of automorphisms of $Q'$ acts on the right on $\mathcal T_{Q'}$ by precomposition, and this action is free and transitive.
Thus, while being a \emph{left} $G$-torsor (respectively a left $H$-torsor), $\mathcal T_{Q'}$ is a \emph{right} $\Aut(Q')$-torsor.
As a consequence, if $\alpha \in \mathcal T_{Q'}$ and $V \in \left(T_{\mathcal T_{Q'}}\right)_{\alpha}$, the tangent vector $V$ can be seen either as an infinitesimal action of $\Aut(Q')$, i.e. an element of $\ad(Q')$, or as an infinitesimal action of $G$, i.e. an element of $\mathfrak g$ (respectively of $\mathfrak h$).
In other words, one has identifications $\ad(Q') \simeq \left(T_{\mathcal T_{Q'}}\right)_{\alpha} \simeq \mathfrak g$ (respectively $\ad(Q') \simeq \left(T_{\mathcal T_{Q'}}\right)_{\alpha} \simeq \mathfrak h$).
The composed identification $\ad(Q') \simeq \mathfrak g$ (respectively $\ad(Q') \simeq \mathfrak h$) is the one given by the trivialization $\alpha$ of $Q'$.
The actions of $G$ (respectively $H$) and $\Aut(Q')$ on $\mathcal T_{Q'}$ commute.
The identifications $\left(T_{\mathcal T_{Q'}}\right)_{\alpha} \simeq \mathfrak g$ identify the tangent bundle $T_{\mathcal T_{Q'}}$ to $\mathfrak g \times \mathcal T_{Q'}$.
The left-action of $G$ on $\mathcal T_{\mathcal Q'}$ induces, by differentiation, a left-action of $G$ on $T_{\mathcal T_{Q'}}$ and thus, by the previous identification, on $\mathfrak g$.
This left-action is the adjoint action of $G$ on $\mathfrak g$.
The $\mathfrak g$-valued $1$-form on $\mathcal T_{Q'}$ given by the identification $T_{\mathcal T_{Q'}} \simeq \mathfrak g \times \mathcal T_{Q'}$ is called the \emph{Maurer-Cartan form}

The global vector fields on $\mathcal T_{Q'}$ that are invariant under (i.e. that commute with) the left action of $G$ (respectively $H$), form a Lie algebra that is canonically identified to $\ad(Q')$.
In particular the flow of such a vector field is given by the action of an element of $\Aut(Q')$.
Conversely, global vector fields that are invariant under the right action of $\Aut(Q')$ form a Lie algebra that is identified to $\mathfrak g$ (respectively $\mathfrak h$), and the flow of such a vector field is given by the action of an element in $G$ (respectively $\mathfrak g$).

The above study can be made in family.
Given a bundle $\pi : \mathcal Q \to M$ of $(G, Q)$-spaces on the manifold $M$, one can consider its bundle of trivializations $\widetilde \pi : \mathcal T_{\mathcal Q} \to M$, whose fiber over $m \in M$ is the $G$-torsor $\mathcal T_{Q_m}$ of trivializations of the fiber $Q_m$ of $\mathcal Q$.
The bundle $\mathcal T_{\mathcal Q}$ is a \emph{$G$-principal bundle}, i.e. a bundle of $G$-torsors.
Reciprocally, given a $G$-principal bundle $\mathcal T$, one gets a bundle of $(G, Q)$-spaces by considering the quotient $G \backslash (\mathcal T \times Q)$.

\begin{remark}
The equivalence between a bundle $\mathcal Q$ of $(G, Q)$-spaces and the associated principal bundle $\mathcal T_{\mathcal Q}$ of trivializations allows to speak indifferently of $\mathcal Q$ or $\mathcal T_{\mathcal Q}$.
In particular, the adjoint and Atiyah bundles of $\mathcal Q$ are also the adjoint and Atiyah bundles of $\mathcal T_{\mathcal Q}$~: $\ad(\mathcal Q) = \ad\left(\mathcal T_{\mathcal Q}\right)$ and $\At(\mathcal Q) = \At\left(\mathcal T_{\mathcal Q}\right)$.
\end{remark}

Suppose given a connection $\nabla$ on $\pi : \mathcal Q \to M$.
It defines a parallel transport on $\mathcal Q$, thus a parallel transport on the bundle of trivializations $\widetilde \pi : \mathcal T_{\mathcal Q} \to M$, that is equivariant with respect to the action of $G$ on $\mathcal T_{\mathcal Q}$.
To such a parallel transport is associated a connection $\widetilde \nabla$, thus a distribution $\left(W_{\alpha}\right)_{\alpha \in \mathcal T_{\mathcal Q}}$ on $\mathcal T_{\mathcal Q}$ with for any $\alpha \in \mathcal T_{\mathcal Q}$, $\left(T_{\mathcal T_{\mathcal Q}}\right)_{\alpha} = W_{\alpha} \oplus \ker (d \widetilde \pi)_{\alpha}$. 
The equivariance of the parallel transport induced by $\widetilde \nabla$ is equivalent to the invariance of the distribution $\left(W_{\alpha}\right)_{\alpha \in \mathcal T_{\mathcal Q}}$ with respect to the action of $G$~: for any $g \in G$ and $\alpha \in \mathcal T_{\mathcal Q}$, $g \cdot W_{\alpha} = W_{g \cdot \alpha}$, where the action of $G$ on the tangent space of $\mathcal T_{\mathcal Q}$ is the differential of the action of $G$ on $\mathcal T_{\mathcal Q}$.

\begin{definition}
A \emph{principal connection} on a $G$-principal bundle $\widetilde \pi : \mathcal T_{\mathcal Q} \to M$ is a distribution $\left(W_{\alpha}\right)_{\alpha \in \mathcal T_{\mathcal Q}}$ that is transverse to the fibers of $\widetilde \pi$ and that is invariant by the action of $G$~: for any $g \in G$, $g \cdot W_{\alpha} = W_{g \cdot \alpha}$.
\end{definition}

\begin{proposition}
The datum of a connection on $\mathcal Q$ is equivalent to the datum of a principal connection on its principal bundle of trivializations $\mathcal T_{\mathcal Q}$.
\end{proposition}

Note that the pullback vector bundle $\widetilde \pi^* \At(\mathcal Q)$ (respectively $\widetilde \pi^* \ad(\mathcal Q)$) has fiber $\left(T_{\mathcal T_{\mathcal Q}}\right)_{\alpha}$ (respectively $\ker (d \widetilde \pi)_{\alpha}$) over $\alpha \in \mathcal T_{\mathcal Q}$.
A principal connection $\left( W_{\alpha} \right)_{\alpha \in \mathcal T_{\mathcal Q}}$ defines for each $\alpha \in \mathcal T_{\mathcal Q}$ a projection $\varpi : \left(T_{\mathcal T_{\mathcal Q}} \right)_{\alpha} \to \ker (d \widetilde \pi)_{\alpha}$, of kernel $W_{\alpha}$.
The projection $\varpi$ is the pullback of the projection $\omega : \At(\mathcal Q) \to \ad(\mathcal Q)$ defined by the connection on $\mathcal Q$ associated to $\left(W_{\alpha}\right)_{\alpha \in \mathcal T_{\mathcal Q}}$.

Since $\ker (d \widetilde \pi)_{\alpha}$ is the tangent space to the fiber of $\mathcal T_{\mathcal Q}$ over $\widetilde \pi(\alpha)$, which is a $G$-torsor, there is an identification $\ker (d \widetilde \pi)_{\alpha} \simeq \mathfrak g$.
In other words, the pullback of $\ad(\mathcal Q)$ by $\widetilde \pi$ is the trivial bundle $\mathfrak g \times \mathcal T_{\mathcal Q}$.
Thus a principal connection can be seen as a map $T_{\mathcal T_{\mathcal Q}} \to \mathfrak g$, that is equivariant for the action of $G$ on $T_{\mathcal T_{\mathcal Q}}$ and the adjoint action of $G$ on $\mathfrak g$.

Let us use this description of principal connections to give another formulation of the notion of Cartan geometry.
Suppose given a bundle $\widetilde{\mathcal Q}$ of $(G, Q, x_0)$-spaces on $M$, and write $\mathcal Q$ the associated bundle of $(G, Q)$-spaces.
We saw that a Cartan geometry on $M$ with underlying bundles $\mathcal Q$ and $\mathcal Q'$ is given by an isomorphism $\omega|_{\At\left(\widetilde {\mathcal Q}\right)} : \At\left(\widetilde {\mathcal Q} \right) \xrightarrow{\sim} \ad(\mathcal Q)$ that is the identity in restriction to $\ad\left(\widetilde {\mathcal Q} \right)$. 
Let $\widetilde{\widetilde \pi} : \mathcal T_{\widetilde {\mathcal Q}} \to M$ be the $H$-principal bundle of trivializations of $\widetilde {\mathcal Q}$, and $\widetilde \pi : \mathcal T_{\mathcal Q}$ the $G$-principal bundle of trivializations of $\mathcal Q$, so that $\mathcal T_{\widetilde {\mathcal Q}} \subset \mathcal T_{\mathcal Q}$. 
One has $\widetilde{\widetilde \pi}^* \At\left(\widetilde {\mathcal Q} \right) = T_{\mathcal T_{\widetilde {\mathcal Q}}}$, $\widetilde{\widetilde \pi}^*\ad(\mathcal Q) = \ker d\widetilde \pi|_{T_{\widetilde {\mathcal Q}}}$ and $\widetilde{\widetilde \pi}^*\ad\left(\widetilde {\mathcal Q}\right) = \ker d \widetilde{\widetilde \pi}$.

We have just noted that $\ker d\widetilde \pi \simeq \mathfrak g \times \mathcal T_{\mathcal Q}$, and similarly $\ker d \widetilde{\widetilde \pi} \simeq \mathfrak h \times \mathcal T_{\widetilde {\mathcal Q}}$.
Thus the pullback of the map $\omega|_{\At\left(\widetilde {\mathcal Q}\right)}$ by $\widetilde{\widetilde \pi}$ can be seen as a $\mathfrak g$-valued $1$-form $\varpi$ on $\mathcal T_{\widetilde {\mathcal Q}}$ that is the Maurer-Cartan form in restriction to the fibers of $\widetilde{\widetilde \pi}$ and such that $\varpi$ is equivariant for the action of $H$ on $\mathcal T_{\widetilde {\mathcal Q}}$ and the adjoint action of $H$ on $\mathfrak g$.
Moreover, for any $\alpha \in \mathcal T_{\widetilde {\mathcal Q}}$, $\varpi_{\alpha}$ is an isomorphism.
This construction can be reversed, so that we have~:

\begin{proposition}
The datum of a Cartan geometry on $M$, modeled on $Q$ is equivalent to the datum of a principal $H$-bundle $\mathcal T$, along with a $\mathfrak g$-valued $1$-form $\varpi$ on $\mathcal T$ such that $\varpi$ is the Maurer-Cartan form in restriction to the fibers of $\mathcal T$, $\varpi$ is equivariant with respect to the action of $H$ on $\mathcal T$ and its adjoint action on $\mathfrak g$, and $\varpi_{\alpha}$ is an isomorphism for all $\alpha \in \mathcal T$.
\end{proposition}

\section{Projective Structures and the Projective Osculating Line}\label{section_projective_structures_projective_osculating_line}

From now on, $X$ is a Riemann surface.
Let $\mathcal O_X$ denote the trivial bundle of $X$, $T_X$ its tangent bundle and $K_X$ its cotangent bundle.
In the rest of the paper, the notions defined in section \ref{section_cartan_geometries} are used in a holomorphic framework, meaning that smoothness conditions are replaced by holomorphicity conditions.
\subsection{Holomorphic Affine Structures and Holomorphic Affine Connections on a Riemann Surface}\label{subsection_affine_structures}

Before coming to projective structures, let us examine the simpler case of \emph{affine structures}, of which the description in terms of Cartan geometries involves very familiar notions of differential geometry.

In this section only, $Q = \C$ and $G = \Aff(\C)$, the group of holomorphic affine transformations acting on $\C$. 
\begin{definition}
A $(\Aff(\C), \C)$-structure on $X$ whose charts are holomorphic is called a \emph{holomorphic affine structure}.
\end{definition}

The affine space $\C$ has a privileged point, namely $x_0 = 0 \in \C$.
The subgroup of $G$ fixing $0$ is the group of holomorphic dilatations $\C^*$.
A bundle of $(G, Q)$-spaces on $X$ is a holomorphic affine bundle on $X$.
A bundle of $(G, Q, x_0)$-spaces on $X$ is a holomorphic line bundle on $X$.
There is a preferred holomorphic line bundle on $X$, namely the tangent bundle $T_X$. 
Let us denote by $\Aff_X$ the associated holomorphic affine bundle, and $s : X \to \Aff_X$ the section of $\Aff_X$ given by the reduction $T_X$, namely the zero section.
Of course, tautologically, for any $x \in X$~: $\left(T_X\right)_x = \left(T_{\Aff_{X, x}}\right)_{s(x)}$.

A Cartan connection $\omega|_{\At\left(T_X\right)}$ on the $(G, Q)$-bundle $\Aff_X$ along with its reduction $T_X$ induces for any $x \in X$ an isomorphism \eqref{eq_iso_exact_cartan_connection} $\Phi_x : \left(T_X\right)_x \xrightarrow{\sim} \left(T_{\Aff_{X, x}}\right)_{s(x)} = \left(T_X\right)_x$, thus an automorphism of $\left(T_X\right)_x$.

\begin{definition}\label{def_affine_connection}
An \emph{affine connection} on $X$ is a Cartan connection on $\Aff_X$ along with its reduction $T_X$, such that for any $x \in X$ the induced automorphism of $\left(T_X\right)_x$ is the identity.
\end{definition}

Now suppose given two affine connections on $X$, namely $\omega|_{\At\left(T_X\right)}$ and $\omega'|_{\At\left(T_X\right)}$.
By definition and the diagram \eqref{eq_iso_exact_cartan_connection}, these two forms coincide when composed with the quotient $\At(T_X) \to \break \At(T_X) / \ad(T_X)$. 
Thus the difference $\omega'|_{\At\left(T_X\right)} - \omega|_{\At\left(T_X\right)}$ takes values in $\ad\left(T_X\right)$.
Denote respectively by $\omega, \omega' : \At\left(\Aff_X\right) \to \ad\left(\Aff_X\right)$ their extensions to $\At\left(\Aff_X\right)$, that are splittings of the Atiyah exact sequence \eqref{eq_ad_at_exact}
\begin{equation}
0 \to \ad\left(\Aff_X\right) \to \At\left(\Aff_X\right) \to T_X \to 0
\end{equation}
and write respectively $\widetilde \omega, \widetilde \omega' : T_X \to \At\left(\Aff_X\right)$ the corresponding splitting morphisms.
The difference $\omega' - \omega$ still takes values in $\ad\left(T_X\right)$, thus so does the difference $\widetilde \omega' - \widetilde \omega$~: it is a holomorphic linear form $\widetilde \omega' - \widetilde \omega : T_X \to \ad\left(T_X\right)$.
Moreover the adjoint bundle of any line bundle is trivial, so $\ad\left(T_X\right) = \mathcal O_X$, the structure sheaf of $X$.
In other words $\widetilde \omega' - \widetilde \omega$ is a section of $K_X$, the cotangent bundle of $X$.
Moreover any section of $K_X$ can be obtained as the difference $\widetilde \omega' - \widetilde \omega$ for some pair of affine connections, and an affine connection $\omega'|_{\At\left(T_X\right)}$ is uniquely determined by the associated splitting morphism $\widetilde \omega$. One has the following proposition~:

\begin{proposition}
The set of affine connections on the Riemann surface $X$ has the structure of an affine space, directed by the vector space of global holomorphic differentials $H^0\left(X, K_X\right)$.
\end{proposition}

Let us have a look at the principal bundles associated to $\Aff_X$ and its reduction $T_X$.
For any $x \in X$, an affine isomorphism $\alpha$ from $\Aff_{X, x}$ to $\C$ is given by a point of $\C$, namely the image of $s(x)$ by $\alpha$, along with a linear isomorphism from the vector space $T_{X, x} = \left(T_{\Aff_X, x}\right)_{s(x)}$ to $\C$, namely the differential of $\alpha$ at $s(x)$.
In other words, a trivilization of the affine space $\Aff_{X, x}$ is given by a $1$-jet of biholomorphism from a neighborhood of $x$ to $\C$.
Thus the $\Aff(\C)$-principal bundle $\paff_X$ associated to $\Aff_X$ has fiber over $x$~:
\begin{equation}
\paff_{X, x} =\left \{j^1_x \phi | \phi \text{ is the germ at $x$ of a local biholomorphism from $X$ to $\C$} \right \}
\end{equation}
The principal bundle $\paff_X$ is also called the \emph{bundle of affine $1$-frames} on $X$.
Now a trivialization of $T_{X, x}$ is given by a trivialization of $\Aff_{X, x}$ sending $s(x)$ to $0$, so the $\C^*$-principal bundle $\mathcal T_X$ associated to $T_X$ has fiber over $x$~:
\begin{equation}
\mathcal T_{X, x} =\left \{j^1_x \phi | \phi \text{ is the germ at $x$ of a local biholomorphism from $X$ to $\C$, $\phi(x) = 0$} \right \}
\end{equation}

Let $(U_i, z_i)_{i \in I}$ be the atlas of a holomorphic affine structure on $X$~: on $U_i \cap U_j$, one has $z_i = g_{ij} \circ z_j$, with $g_{ij} \in \Aff(\C)$.
For each $i \in I$, one gets a local holomorphic section (thus trivialization) of $\paff_X|_{U_i}$ by considering $x \in U_i \mapsto j^1_xz_i \in \paff_{X, x}$.
Moreover for any $i, j \in I$, on the intersection $U_i \cap U_j$, the change of trivialization is the constant $g_{ij} \in \Aff(\C)$, thus the local connections on each $U_i$ associated to the trivializations defined by the charts $z_i$ glue together on the intersections $U_i \cap U_j$.
As a consequence, an affine structure on $X$ defines a global connection on $\paff_X$.

A straightforward calculation, similar to the proofs of proposition \ref{atiyah_1jet} and lemma \ref{derivee_schwarzienne} below, shows that the connection on $\paff_X$ given by an affine structure is an affine connection, and that any affine connection comes from a unique affine structure, so that the following proposition holds~:
\begin{proposition}
The datum of a holomorphic affine structure on the Riemann surface $X$ is equivalent to the datum of an affine connection on $X$.
In particular, the set of affine structures on $X$ is an affine space, directed by $H^0(X, K_X)$.
\end{proposition}

\subsection{The Principal Bundle of Projective $2$-Frames of a Riemann Surface}

From now on, the model space we consider is the Riemann sphere $Q = \CP^1$, and its group of symmetries is $G = \PSL(2, \C)$. The Riemann sphere $\CP^1$ has a preferred point $x_0 = 0 \in \CP^1$.
Let $H \subset G$ be the subgroup whose action on $\CP^1$ fixes $0$.
With these conventions, a bundle of  $(G, Q)$-spaces is a holomorphic $\CP^1$-bundle and a bundle of $(G, Q, x_0)$-spaces is a holomorphic $\CP^1$-bundle along with a holomorphic section.

Let also $L \subset H$ be the subgroup whose action on the tangent space of $\CP^1$ at $0$ is trivial.
Let $\mathfrak l \subset \mathfrak h \subset \mathfrak g$ be the respective Lie algebras of $L \subset H \subset G$.
The Lie subalgebra $\mathfrak h \subset \mathfrak g$ is not preserved by the adjoint action of $G$. 
However, the adjoint action of $H$ over $\mathfrak h$ preserves $\mathfrak l$.

The action of $G$ on $\CP^1$ induces an isomorphism of Lie algebras $\mathfrak g \simeq H^0(\CP^1, T_{\CP^1})$, where the Lie bracket on $H^0(\CP^1, T_{\CP^1})$ is given by the Lie bracket of vector fields. The Lie subalgebra $\mathfrak h$ is identified with the space of vector fields that vanish at $0 \in \CP^1$ and $\mathfrak l$ is identified with the space of vector fields whose $1$-jet vanishes at $0$.

If $x \in \CP^1$, denote by $(\CP^1, x)$ the germ of $\CP^1$ at $x$.
Similarly to what happens for affine transformations of $\C$, it is a classical fact (see for instance \cite{deligne_2}) that for any $x \in \CP^1$ and any $2$-jet $j$ at $x$ of local biholomorphism with values in $\CP^1$, there is exactly one Möbius transformation in $G$ whose $2$-jet at $x$ is $j$.
This justifies the following definition.
\begin{definition}
The \emph{bundle of projective $2$-frames} on $X$, denoted by $\pi : \mathcal P_X \twoheadrightarrow X$, has fiber over $x \in X$~:
\begin{equation}
\mathcal P_{X, x} =\left \{j^2_x \phi | \phi \text{ is the germ at $x$ of a local biholomorphism from $X$ to $\CP^1$} \right \}
\end{equation}
The holomorphic structure of $\mathcal P_X$ is such that if $\phi$ is a local holomorphic function on $X$, then $x \in X \mapsto j^2_x\phi \in \mathcal P_{X, x}$ is a holomorphic section of $\mathcal P_X$.
\end{definition}
The group $G$ acts on the left on $\mathcal P_X$ by postcomposition, and by the above discussion the action is free and transitive on the fibers, thus $\mathcal P_X$ is a $G$-principal bundle.

Following \eqref{eq_bundle_as_quotient_principal} and remark \ref{rk_quotient_principal_by_isotropies}, one can consider the $\CP^1$-bundle whose bundle of trivializations is $\mathcal P_X$~:
\begin{definition}
The \emph{projective osculating line} of the Riemann surface $X$ is the $\CP^1$-bundle over $X$, denoted by $\Pi : P_X \to X$ and defined by $P_X = G \backslash \left(\CP^1 \times \mathcal P_X \right) = H \backslash \mathcal P_X$.
\end{definition}

Having chosen $x_ = 0 \in \CP^1$ as  a preferred point, the principal bundle $\mathcal P_X$ is endowed with a subbundle $\mathcal S_X \subset \mathcal P_X$, that is a $H$-principal bundle, and whose fiber over a point $x \in X$ is defined in the following way~:
\begin{equation}
\mathcal S_{X, x} =\left \{j^2_x \phi | \phi \text{ germ at $x$ of local biholomorphism from $X$ to $\CP^1$}, \phi(x) = 0 \right \}
\end{equation}
In other words, $\mathcal S_X$ is the bundle of projective $2$-frames on $X$ that send the points of $X$ to $0 \in \CP^1$.
The subbundle $\mathcal S_X$ is the bundle of trivializations of a $(G, Q, x_0)$-bundle whose underlying $(G, Q)$-bundle is $P_X$.
In other words, $\mathcal S_X$  defines a holomorphic section $s_X: X \to P_X$, given by $s(X) = H \backslash \mathcal S_X \subset H \backslash \mathcal P_X = P_X$.

\subsection{The Structure of Filtered $\SO(3, \C)$-Bundle on $\ad(\mathcal P_X)$}

Recall the definition of a holomorphic filtered $\SO(3, \C)$-bundle on $X$.
\begin{definition}
A \emph{filtered $\SO(3, \C)$-bundle} on $X$ is a rank $3$ holomorphic vector bundle on $X$ endowed with
\begin{itemize}
\item[(i)] A nondegenerate bilinear form $B_x$ on each fiber $W_x$ of $W$, such that $B_x$ varies holomorphically with $x \in X$.
\item[(ii)] An identification $\bigwedge^3 W \simeq \mathcal O_X$ such that the bilinear form induced by $B$ on the fibers of $\bigwedge^3 W$ is the trivial one on $\mathcal O_X$
\item[(iii)] A filtration $F_1 \subset F_2 \subset W$ along with identifications $F_1 \simeq K_X$ and $F_2 / F_1 \simeq \mathcal O_X$ such that for all $x \in X$, $F_2$ is the orthogonal of $F_1$ for the nondegenerate bilinear form $B_x$.
\end{itemize}
\end{definition}
\begin{remark}
Condition $(iii)$ implies that $B_x$ induces a perfect pairing between $F_1$ and $W / F_2$, thus $W / F_2$ is identified with $T_X$.
\end{remark}

It is shown in \cite{biswas-dumitrescu} that the datum of a holomorphic $\CP^1$-bundle with a holomorphic section is equivalent to the datum of a holomorphic filtered $\SO(3, \C)$-bundle.
In this subsection we exhibit the filtered $\SO(3, \C)$-bundle associated to the projective osculating line $P_X$ along with its canonical section $s_X$, and we show that this particular $\SO(3, \C)$ carries slightly more structure.

The fiber over $x \in X$ of the adjoint bundle $\ad(\mathcal P_X)$ ($= \ad\left(P_X\right)$) is the space of holomorphic vector fields over the projective line $P_{X, x}$.
The adjoint bundle $\ad(\mathcal P_X)$ is a vector bundle of rank $3$ and it is endowed with a filtration $F_1^X \subset F_2^X \subset \ad(\mathcal P_X)$, defined as follows.
For any $x \in X$, take for $F^X_{2, x}$ (respectively for $F^X_{1, x}$)  the $2$-dimensional (resp. $1$-dimensional) vector space whose elements are vector fields on the projective line $P_{X, x}$ that vanish at $s(x) \in P_{X, x}$ (resp. whose $1$-jet vanishes at $s(x)$).
Note that $F_2^X = \ad(\mathcal S_X)$.

For $x \in X$, since $\mathcal P_X$ is the bundle of trivializations of $P_X$, the datum of an element $\overline \phi \in \mathcal P_{X, x}$ ($\phi$ is a local biholomorphism from a neighborhood of $x$ to $\CP^1$, and the bar denotes the class in $\mathcal P_{X, x}$) is the same as an identification $P_{X, x} \simeq \CP^1$, and induces an identification $\ad(\mathcal P_X) \simeq \mathfrak g$.
If moreover $\overline \phi \in \mathcal S_{X, x}$, then $\mathcal S_{X, x}$ is identified with $H$, $F_{2, x}^X$ with $\mathfrak h$ and $F_{1, x}^X$ with $\mathfrak l$.

For any $x \in X$, denote by $J^2(T_X)_x$ the $3$-dimensional vector space of $2$-jets at $x$ of local vector fields on $X$.
The holomorphic vector bundle $J^2(T_X)$ comes with a filtration $\widetilde F_1^X \subset \widetilde F_2^X \subset J^2(T_X)$, where $\widetilde F_{2, x}^X$ (respectively $\widetilde F_{1, x}^X$) contains the $2$-jets of vector fields vanishing (respectively with vanishing $1$-jet) at the point $x$.
The line bundle $\widetilde F_1^X$ is identified with $K_X^{\otimes 2} \otimes T_X = K_X$.
The line bundle $\widetilde F_2^X / \widetilde F_1^X$ is identified with $K_X \otimes T_X = \mathcal O_X$.
The line bundle $J^2(T_X)/\widetilde F_{2, x}$ is identified with $T_X$.

In the case of the model space $\CP^1$, since any $2$-jet of local vector field at a point $x \in \CP^1$ uniquely extends to a global vector field in $H^0(\CP^1, T_{\CP^1})$, there is a canonical identification $J^2(T_{\CP^1})_x \simeq H^0(\CP^1, T_{\CP^1})$.
Note that in this way $\widetilde F_{2, x}^{\CP^1}$ (respectively $\widetilde F_{1, x}^{\CP^1}$) is identified with $H^0\left(\CP^1, T_{\CP^1}\left(-[x]\right)\right)$ (respectively\break $H^0\left(\CP^1, T_{\CP^1}\left(-2[x]\right)\right)$).

\begin{proposition}\label{2jet}
There is a canonical identification $F_2^X \simeq \widetilde F_2^X$, that sends moreover $F_1^X$ to $\widetilde F_1^X$.
\end{proposition}

\begin{proof}
Let $x \in X$ and $\overline \phi \in \mathcal S_{X, x}$. In particular $\phi(x) = 0$.
The local biholomorphism $\phi$ thus identifies a neighborhood of $x \in X$ with a neighborhood of $0 \in \CP^1$, and as a consequence it defines an isomorphism $\delta \phi : J^2(T_X)_x \simeq J^2(T_{\CP^1})_0$.
Let us explicitely write down this isomorphism.

Let $z$ be a local coordinate on $X$ centered at $x$, and let $\varphi$ be the holomorphic function on a neighborhood of $0 \in \C$ such that $\phi = \varphi(z)$.
The vector space $J^2(T_{\C})_0$ is identified with $\C^3$~: if $f$ is a function defined in a neighborhood of $0 \in \C$, then the $2$-jet $j^2_0f \in $ is given by $(f(0), f'(0), f''(0))$.
As a consequence, the coordinate $z$ identifies $J^2(T_X)_x$ with $\C^3$.
Using this identification, one has $\delta \phi(a_0, a_1, a_2)_z = (\varphi'a_0, \varphi''a_0 + \varphi'a_1, \varphi'''a_0 + 2\varphi''a_1 + \varphi'a_2)$.

In particular, if $(a_0, a_1, a_2) \in \widetilde F_{2, x}^X \subset J^2(T_X)_x$, i.e. $a_0 = 0$, then $\delta \phi(a_0, a_1, a_2) \in \widetilde F_{2, 0}^{\CP^1}$, and moreover $\delta \phi(a_0, a_1, a_2)$ only depends on $j^2_x \phi = \overline \phi$.
Thus $\overline \phi$ defines an isomorphism $\delta \overline \phi$ between $\widetilde F_{2, x}^X$ and $\widetilde F_{2, 0}^{\CP^1} = H^0\left(\CP^1, T_{\CP^1}\left(-[x]\right)\right) = \mathfrak h$.
Clearly, if $\psi \in \mathcal S_{X, x}$, then $\phi = h \cdot \phi$ with $h \in H$, and $\delta \overline \psi = h \cdot \delta \overline \phi$ where the action of $h$ on $\mathfrak h$ is the adjoint action.
Note finally that if $(a_0, a_1, a_2) \in \widetilde F_{1, x}^X \subset J^2(T_X)_x$, i.e. $a_0 = a_1 = 0$, then $\delta \overline \phi(a_0, a_1, a_2) \in \mathfrak l$.

On the other hand, we already noted that $\overline \phi$ gives an isomorphism $\Delta \overline \phi : F_{2, x}^X \simeq \mathfrak h$, that sends $F^X_{1, x}$ to $\mathfrak l$.
We also have $\Delta \overline \psi = h \cdot \Delta \overline \phi$.

Finally, $\overline \phi$ gives an isomorphism $\delta \overline \phi^{-1} \circ \Delta \overline \phi : F_{2, x}^X \simeq \widetilde F_{2, x}^X$.
We have $\delta \overline \psi^{-1} \circ \Delta \overline \psi = (h \cdot \delta \overline \phi^{-1}) \circ (h \cdot \Delta \overline \phi) = \delta \overline \phi^{-1} \circ \Delta \overline \phi$.
Thus the isomorphism $F_{2, x}^X \simeq \widetilde F_{2, x}^X$ does not depend on the choice of $\overline \phi$.
It sends $F_{1, x}^X$ to $\widetilde F_{1, x}^X$, since $\delta \overline \phi$ (resp. $\Delta \overline \phi$) sends $\widetilde F_{1, x}^X$ (resp. $F_{1, x}^X$) to $\mathfrak l$.
It is clear that the induced map $F_2^X \simeq \widetilde F_2^X$ is holomorphic.
\end{proof}

\begin{proposition}\label{1jet}
There is a canonical identification $\ad(\mathcal P_X) / F_1^X \simeq J^2(T_X) / \widetilde F_1^X = J^1(T_X)$, that sends moreover $F_2^X / F_1^X$ to $\widetilde F_2^X / \widetilde F_1^X$.
\end{proposition}

\begin{proof}
The proof is similar to proposition \ref{2jet}.
On the one hand, given $x \in X$ and $\overline \phi \in \mathcal S_{X, x}$, $\overline \phi$ identifies $J^1(T_X)_x$ to $J^1(T_{\CP^1})_0 = \mathfrak g / \mathfrak l$.
On the other hand, $\overline \phi$ identifies $\ad(\mathcal P_X)_x / F_{1, x}^X$ to $\mathfrak g / \mathfrak l$. 
Thus it defines an isomorphism $J^1(T_X)_x \simeq \ad(\mathcal P_X)_x / F_{1, x}^X$, that happens to not depend on the choice of $\overline \phi$, and to send $F_{2, x}^X / F_{1, x}^X$ to $\widetilde F_{2, x}^X / \widetilde F_{1, x}^X$.
Finally, the induced map from $\ad(\mathcal P_X) / F_1^X$ to $J^1(T_X)$ is holomorphic.
\end{proof}

\begin{remark}\label{intrinseque}
Let $Y$ be another Riemann surface and $f : X \simeq Y$ a biholomorphism.
The map $f$ induces identifications $\iota_1 : \ad(\mathcal S_X) \simeq f^* \ad(\mathcal S_Y)$, $\iota_2 : \ad(\mathcal P_X)/F_1^X \simeq f^* \ad(\mathcal P_Y)/F_1^Y$, $\iota_3 : J^1(T_X) \simeq f^*J^1(T_Y)$ and $\iota_4 : \widetilde F_2^X \simeq f^*\widetilde F_2^Y$.
From the proofs of propositions \ref{2jet} and \ref{1jet}, it appears that the following diagram commutes, where the horizontal arrows are given by the isomorphisms in propositions \ref{2jet} and \ref{1jet}~:

\begin{multicols}{2}
\begin{equation}\label{intrinseque_2jet}
\begin{tikzcd}
\ad(\mathcal S_X) \arrow[r] \arrow[d, "\iota_1"] &\widetilde F_2^X \arrow[d, "\iota_4"] \\ f^*\ad(\mathcal S_Y) \arrow[r] &  f^*\widetilde F_2^Y
\end{tikzcd}
\end{equation}

\begin{equation}\label{intrinseque_1jet}
\begin{tikzcd}
\ad(\mathcal P_X)/F_1^X \arrow[r] \arrow[d, "\iota_2"] & J^1(T_X) \arrow[d, "\iota_3"] \\ f^* \ad(\mathcal P_Y)/F_1^Y \arrow[r] &  f^*J^1(T_Y)
\end{tikzcd}
\end{equation}
\end{multicols}
\end{remark}

We have, as a consequence of propositions \ref{2jet} and \ref{1jet}~:

\begin{corollary}\label{tangent}
There are canonical identifications~: 
\begin{itemize} 
\item[(i)] $\ad(\mathcal P_X)/F_2^X \simeq T_X$ 
\item[(ii)] $F_2^X / F_1^X \simeq \mathcal O_X$ 
\item[(iii)] $F_1^X \simeq K_X$ 
\end{itemize}
\end{corollary}

This corollary, along with the remark that, for any $x \in X$, $F_{2, x}^X$ is the orthogonal of $F_{1, x}^X$ for the killing form of the Lie algebra $\ad(\mathcal P_X)_x$, shows that $\ad(\mathcal P_X)$ is canonically endowed with a structure of a filtered $\SO(3, \C)$-bundle.

\subsection{Projective Connections}

In this subsection we introduce a special class of Cartan connections on $\mathcal P_X$ along with its reduction $\mathcal S_X$, namely \emph{projective connections}.
To do that we use propositions \ref{2jet} and \ref{1jet}, as well as the following proposition about the Atiyah bundle $\At(\mathcal S_X)$ of $\mathcal S_X$.

\begin{proposition}\label{atiyah_1jet}
There is an isomorphism $\Phi_X : \At(\mathcal S_X)/F_1^X \xrightarrow{\sim} \ad(\mathcal P_X)/F_1^X = J^1 (T_X)$ such that the following diagram commutes~: 
\begin{equation}\label{iso_atiyah}
\begin{tikzcd} \At(\mathcal S_X)/F_1^X \arrow[d, twoheadrightarrow, "\widetilde \pi"] \arrow[r, "\Phi_X"] &\ad(\mathcal P_X)/F_1^X = j^1(T_X) \arrow[d, twoheadrightarrow]\\ T_X \arrow[r, equal] & T_X \end{tikzcd}
\end{equation}
\end{proposition}

\begin{proof}
Let $x \in X$ and $z_1$ be a local coordinate on a neighborhood $U \subset X$ of $x$.
The coordinate $z_1$ defines a local section $\gamma_{z_1} = j^2z_1$ of the principal bundle $\mathcal P_X$, and thus trivializations of $\mathcal P_X|_U$ and $P_X|_U$.
In particular $\gamma_{z_1}$ gives a principal connection on $\mathcal P_X|_U$, i.e. a local morphism $\phi_{z_1} : \At(\mathcal P_X) \to \ad(\mathcal P_X)$ that is the identity on $\ad(\mathcal P_X) \subset \At(\mathcal P_X)$.
Since $F_1^X \subset \ad(\mathcal P_X)$, the principal connection $\phi_{z_1}$ induces a local morphism $\At(\mathcal P_X)/F_1^X \to \ad(\mathcal P_X)/F_1^X$, whose restriction to $\At(\mathcal S_X) \subset \At(\mathcal P_X)$ we write $\Phi_X^{z_1}$.

Let us show that diagram \ref{iso_atiyah} commutes when $\Phi_X$ is replaced by $\Phi_X^{z_1}$.
Let $\mu_{z_1} : P_X|_U \to \CP^1$ be the trivialization of $P_X$ defined by $\gamma_{z_1}$ : if $\Pi : P_X \to X$ is the obvious projection map, then $(\Pi, \mu|_{z_1})$ identifies $P_X|_U$ to $U \times \CP^1$.
Then it is not hard to see that the expression $\mu_{z_1} \circ s_X : U \to \CP^1$ of section $s$ in the trivialization $\gamma_{z_1}$ is given by $\mu_{z_1} \circ s_X = z_1$.
Let $V \in \Gamma(U, \At(\mathcal S_X))$ and $[V]$ its class modulo $F_1^X$.
On the one hand we have $dz_1 \circ \varpi \circ \Phi_X^{z_1}([V]) = d\mu_{z_1}(V(s)) \in \Gamma(U, z_1^*T_{\CP^1})$, where we see $V$ as a vector field on $P_X|_U$, and $\varpi$ is the projection $J^1(T_X) \to T_X$.
On the other hand, for any $y \in U$ and any $p \in P_{X, y}$, $\widetilde \pi([V])_y = d\Pi(V(p))$ where in the second term $V$ is seen as a vector field on $P_X|_U$.
In particular, choosing $p = s(y)$, we have $\widetilde \pi([V]) = d\Pi(V(s))$.
Since $V \in \Gamma(U, \At(\mathcal S_X))$, $V$ is tangent to $s$, i.e. $V(s) = ds \circ d\Pi(V(s))$.
Thus $d\mu_{z_1}(V(s)) = d\mu_{z_1} \circ ds \circ d\Pi(V(s))$ and by using $\mu_{z_1} \circ s_X = z_1$~: $d\mu_{z_1}(V(s)) = dz_1 \circ d\Pi(V)$.
Finally, since $dz_1$ is an isomorphism, $\varpi \circ \Phi_X^{z_1} = d\Pi$.

It follows that $\Phi_X^{z_1}$ is an isomorphism.
Indeed, it is the identity on $\ad(\mathcal S_X) / F_1^X$ and $T_X$, and both $\At(\mathcal S_X)/F_1^X$ and $\ad(\mathcal P_X)/F_1^X = j^1(T_X)$ are extensions of $T_X$ by $\ad(\mathcal S_X) / F_1^X$.

It remains to prove that if $z_2$ is another local coordinate on $U$, then $\Phi_X^{z_2} = \Phi_X^{z_1}$.
If it is the case, then the $\Phi_X^z$ for all local coordinates on $X$ glue together in a global isomorphism $\Phi_X$.
Write $\gamma_{z_2} = g \cdot \gamma_{z_1}$, where $g : U \to G$ is a holomorphic function.
It is a consequence of a computation in the proof of lemma \ref{derivee_schwarzienne} that for any $y \in U$ and any $V \in T_yX$, the vector field $dg_y(V) \in H^0(\CP^1, T_{\CP^1}) = \mathfrak g$ has vanishing $1$-jet at $z_1(y) \in \CP^1$.
As a consequence, $d\mu_{z_2} - g \cdot d\mu_{z_1}$ has values in $d\lambda^{-1} F_1^X$, where $\lambda : \mathcal P_X \to \At(\mathcal P_X)$ is the quotient map.
This shows that $\phi_{z_2} - \phi_{z_1}$ has values in $F_1^X$, and thus $\Phi_X^{z_2} = \Phi_X^{z_1}$.
\end{proof}

Take a Cartan connection for the bundles $\mathcal S_X \subset \mathcal P_X$, that we see as a splitting $\omega : \At(\mathcal P_X) \to \ad(\mathcal P_X)$ of the short exact sequence $0 \to \ad(\mathcal P_X) \to \At(\mathcal P_X) \to T_X \to 0$ such that the induced morphism $\omega|_{\At(\mathcal S_X)} : \At(\mathcal S_X) \to \ad(\mathcal P_X)$ is an isomorphism.
The lines of the following diagram are exact~:

\begin{equation}\label{diag_connexion_proj}
\begin{tikzcd} 0 \arrow[r] &F_1^X \arrow[d, equal] \arrow[r] \arrow[dr, phantom, "\circlearrowleft"] & \At(\mathcal S_X) \arrow[d, "\omega|_{F_2^X}"] \arrow[r] & \At(\mathcal S_X)/F_1^X \arrow[d, "\Phi_X"] \arrow[r] & 0 \\ 0 \arrow[r] & F_1^X \arrow [r] & \ad(\mathcal P_X) \arrow[r] & \ad(\mathcal P_X) / F_1^X = J^1(T_X) \arrow[r]& 0 \end{tikzcd}
\end{equation}

The left square of \eqref{diag_connexion_proj} commutes for any Cartan connection $\omega$.
This is however not the case for the right corner.

\begin{definition}
A \emph{projective connection} on $X$ is a Cartan connection $\omega$ for the principal bundles $\mathcal S_X \subset \mathcal P_X$, such that diagram \eqref{diag_connexion_proj} commutes.
\end{definition}

\begin{remark}
The vector bundle $\ad(\mathcal P_X)$, along with the filtration $F_1^X \subset F_2^X \subset \ad(\mathcal P_X)$ and a projective connection, is a \emph{$\SO(3, \mathbb C)$-oper}, in the sense of \cite{biswas-dumitrescu}.
\end{remark}
\begin{remark}
The datum of a Cartan connection on $\mathcal P_X$ is equivalent to the datum of a connection on the projective osculating line $P_X$, thus a foliation of $P_X$ transverse to the fibers.
The projective connections give rise to a special class of such foliations~: two foliations associated to projective connections are tangent to order $2$ at the point $s_X(x) \in P_X$, for any $x \in X$.
\end{remark}

Take $\omega_1, \omega_2$ two projective connections on $X$, $x \in X$ and $v \in \At(\mathcal P_X)_x$. There exists $v_1 \in \ad(\mathcal P_X)_x$ and $v_2 \in \At(\mathcal S_X)_x$ such that $v = v_1 + v_2$ ($v_1$ and $v_2$ are not unique).
One has ($i = 1, 2$)~: $\omega_i(v_1) = v_1$ and $[\omega_i(v_2)] = \Phi_X([v_2]) \in J^1(T_X)_x$, where the brackets stand for the class modulo $F_1^X$.
As a consequence, $(\omega_1 - \omega_2)_x(v) \in F_{1, x}$.
Moreover, since $\omega_i$ splits the exact sequence $0 \to \ad(\mathcal P_X) \to \At(\mathcal P_X) \to T_X \to 0$, it is also given by a morphism $\widetilde \omega_i : T_X \to \At(\mathcal P_X)$.
For any $w \in T_{X,x}$, $(\widetilde \omega_{2, x} - \widetilde \omega_{1, x})(w) = (\omega_1 - \omega_2)(v) \in F_{1, x}$ (where $v \in \At(\mathcal P_X)_x$ such that $\widetilde \pi(v) = w$), so $\widetilde \omega_2 - \widetilde \omega_1$ defines a section of $\Hom\left(T_X, F_1^X\right)$, i.e. of $K_X^{\otimes 2}$ since $F_1^X = K_X$ by corollary \ref{tangent}.
Reciprocally, if $\widetilde \omega : T_X \to \At(\mathcal P_X)$ defines a projective connection and $\varphi \in H^0\left(X, K_X^{\otimes 2}\right)$, then $\widetilde \omega + \phi$ defines a projective connection as well.
Thus one has the following proposition~:

\begin{proposition}\label{structure_affine}
The set of projective connections on $X$ is an affine space directed by $H^0\left(X, K_X^{\otimes 2}\right)$.
\end{proposition}

\subsection{Projective Connections and Projective Structures}\label{subsection_connexions_projectives_structures_projectives}

In this section, we show that the datum of a projective connection on $X$ is the same as the datum of a \emph{projective structure} on $X$, i.e. a $(\PGL(2, \C), \CP^1)$-structure on $X$.
Let us recall the definition in this precise case.

\begin{definition}
A \emph{projective atlas} on $X$ is the datum of an open cover $(U_i)_{i \in I}$ and local charts $z_i : U_i \to \CP^1$ such that for any $i, j \in I$, there exists $g_{ij} \in \PGL(2, \mathbb C)$ such that $z_i = g_{ij} \circ z_j$ on $U_i \cap U_j$.

Two projective atlases are said to be equivalent if their union is a projective atlas.
A \emph{projective structure} on $X$ is an equivalence class of projective atlases.
\end{definition}

Take a projective structure of atlas $(U_i, z_i)_{i \in I}$.
Each $z_i$ defines a principal connection on $\mathcal P_X$, thus a splitting $\omega_i$ over $U_i$ of the exact sequence $0 \to \ad(\mathcal P_X) \to \At(\mathcal P_X) \to T_X \to 0$ ~: the section $x \in U_i \mapsto j^2_x z_i \in \mathcal P_{X, x}$ is flat.
Since $z_i = g_{ij} \circ z_i$, where $g_{ij}$ is a constant in $G$, the induced connections, thus the $\omega_i$'s, coincide on the intersections $U_i \cap U_j$, and thus define a global splitting $\omega$.
The proof of proposition \ref{atiyah_1jet} shows that $\omega$ is a projective connection.
Thus a projective structure on $X$ defines a projective connection on $X$.

Let $f : U \subset \C \to \C$ be a local biholomorphism.
The \emph{Schwarzian derivative} of $f$, denoted $S(f)$, is the function on $U$ defined as follows~: 

\begin{equation}  
S(f) = \left (\frac{f''}{f'} \right) - \frac 12 \left(\frac{f''}{f'}\right)^2 
\end{equation}

One has $S(f) = 0$ if and only if $f$ is an element of $PGL(2, \mathbb C)$.
If $f_1, f_2$ are two functions defined on open sets of $\C$ such that $f_2 \circ f_1$ has nonempty domain, one has the following formula~:

\begin{equation}\label{der_schw_composee}
S(f_2 \circ f_1) = \left (S(f_2) \circ f_1 \right) (f_1')^2 + S(f_1)
\end{equation}

Let $z_1, z_2$ be holomorphic coordinates on an open subset $U \subset X$.
We denote by $\{z_2, z_1\}$ the holomorphic function on $U$ defined by $\{z_2, z_1\} =S(f)$, where $f$ is the function on $z_1(U) \subset \mathbb C$ such that $z_2 = f(z_1)$.
Formula \eqref{der_schw_composee} means that if $z_3$ is another local coordinate on $U$, then one has $\{z_3, z_1\}dz_1^{\otimes 2} = \{z_3, z_2\}dz_2^{\otimes 2} + \{z_2, z_1\}dz_1^{\otimes 2}$.
For details on the properties of the schwarzian derivative, see \cite{hubbard}.

The two local coordinates $z_1, z_2$ are charts of two projective structures $p_1, p_2$ on $U$.
The projective structures $p_1$ and $p_2$ in turn define two projective connections $\widetilde \omega_1, \widetilde \omega_2$ on $U$.
According to \ref{structure_affine}, the difference $\widetilde \omega_2 - \widetilde \omega_1$ is given by a quadratic differential on $U$~: $\widetilde \omega_2 - \widetilde \omega_1 \in H^0\left(U, K_X|_U^{\otimes 2}\right)$.

\begin{proposition}\label{lien_derivee_schwarzienne}
Using the previous notations, one has $\widetilde \omega_2 - \widetilde \omega_1 = \{z_2, z_1\}dz_1^{\otimes 2}$.
\end{proposition}

The following lemma provides an interpretation of the schwarzian derivative as an actual derivative.
See \cite{anderson} for more details.

\begin{lemma}\label{derivee_schwarzienne}
Consider $f : U \subset \mathbb C \to \mathbb C$ a local biholomorphism.
Let $g : U \to PGL(2, \mathbb C)$ be the holomorphic map such that for any $t \in U$, $j^2_t(g(t)) = j^2_tf$, i.e. $g(t)(t) = f(t), g(t)'(t) = f'(t), g(t)''(t) = f''(t)$. 

The derivative $\frac{dg(t)}{dt}$ is a map from $U$ to $\mathfrak{sl}_{2, \mathbb C}$.
After identifying $\mathfrak{sl}_{2, \mathbb C}$ with $H^0(\CP^1, T_{\CP^1})$, one has $$\left. \frac{dg(t)}{dt}\right |_{t = t_0} =S(f)(t_0) \frac{(t-t_0)^2}{2} \frac{\partial}{\partial t}$$
\end{lemma}

In the previous statement, $\CP^1$ is seen as $\CP^1 = \C \cup \{\infty\}$, and $t$ is the coordinate on $\C$.
Thus a basis of $H^0(\CP^1, T_{\CP^1})$ is given by the global vector fields $\frac{\partial}{\partial t}$, $(t-t_0) \frac{\partial}{\partial t}$ and $\frac{(t-t_0)^2}2 \frac{\partial}{\partial t}$

\begin{proof}
A straightforward computation shows that for any $t \in U$, $g(t)$ is given by the the matrix 
$$g(t) \equiv \begin{pmatrix} a(t) & b(t) \\ c(t) & d(t) \end{pmatrix}$$
with 
\begin{equation}
\label{osculation}
\begin{aligned} a &= ff'' -2(f')^2 \\ b &= -2ff' -t(ff'' - 2(f')^2) \\ c &= f'' \\ d &= -2f' - tf'' \end{aligned}
\end{equation}

With these notations, the derivative of $g$ at $t_0$, which is a global vector field on $\CP^1$, is given by
\begin{equation}\label{derivation_PGL2}
\begin{aligned}g'(t_0) =& \frac1{a(t_0)d(t_0) - b(t_0)c(t_0)}((a(t_0)c'(t_0)-a'(t_0)c(t_0))t^2 \\&+ (a'(t_0)d(t_0)-b(t_0)c'(t_0) -a(t_0)d'(t_0)+b'(t_0)c(t_0))t\\& + (b'(t_0)d(t_0)-d'(t_0)b(t_0)))\frac{\partial}{\partial t}\end{aligned}
\end{equation}

Lemma \ref{derivee_schwarzienne} is obtained by combining formulas \eqref{osculation} and \eqref{derivation_PGL2}.
\end{proof}

\begin{proof}(of proposition \ref{lien_derivee_schwarzienne})
Write $z_2 = f(z_1)$, and let $\gamma_{z_1}, \gamma_{z_2}$ be the local sections of $\mathcal P_X$ defined by $z_1, z_2$. One has $\gamma_{z_2} = g \cdot z_2$, where for any $x \in U$, $j^2_{z_1(x)}g(z_1(x)) = j^2_{z_1(x)}f$ and $g(z_1(x)) \in G$.

Let $\widetilde \omega_1, \widetilde \omega_2 : T_X|_U \to \At(\mathcal P_X)|_U$ be the morphisms associated to the projective connections defined by $z_1$ and $z_2$~: $\widetilde \omega_1$ is given by $d\gamma_{z_1}$. In the coordinate $z_1$, $\left(\widetilde \omega_2 - \widetilde \omega_1\right)\left(\frac{\partial}{\partial z_1}\right) \in \mathfrak{sl}_{2, \mathbb C}$ is given by $g'(z_1)$.
Lemma \ref{derivee_schwarzienne} then implies $(\widetilde \omega_2 - \widetilde \omega_1)|_{z_1(x)}\left(\frac{\partial}{\partial z_1}\right) = \{z_2, z_1\}(x)\frac{(z_1 - z_1(x))^2}2 \frac{\partial}{\partial z_1}$. Since $\left.\frac d {dz_1} \frac{(z_1 - z_1(x))^2}2 \right |_{z_1 = z_1(x)} = 0$ and $\frac {d^2} {dz_1^2} \frac{(z_1 - z_1(x))^2}2 = 1$, it is a consequence of proposition \ref{2jet} that $dz_1$, which is a section of $K_X \simeq F_1^X$, is identified with $\frac12 (z_1 - z_1(x))^2\frac{\partial}{\partial z_1} \in \ad(S_X)$.
This ends the proof.
\end{proof}

\begin{corollary}
Any projective connection on $X$ comes from a projective structure on $X$.
\end{corollary}

\begin{proof}\label{coro_equiv_conn_struct}
This is a consequence of proposition \ref{lien_derivee_schwarzienne} and of the fact that the equation $S(f) = h$ has local solutions for any holomorphic function $h$ defined on an open subset of $\C$ (for this last fact, see for instance \cite{saint-gervais}).
\end{proof}

Thus the datum of a projective structure on $X$ is equivalent to the datum of a projective connection on $X$. 
Proposition \ref{structure_affine} gives~:

\begin{proposition}
The set of projective structures on $X$ is an affine space directed by the vector space $H^0(X, K_X^{\otimes 2})$ of global quadratic differentials on $X$.
\end{proposition}

\section{Branched Projective Structures and Branched Projective Osculating Lines}\label{section_bps_bol}

\subsection{Ramified Coverings and $PGL(2, \mathbb C)$-Action}\label{revet_action}

Let $(U, x)$ be a germ of Riemann surface, identified with $(V, 0)$ where $V \subset \C$ is a neighborhood of $0$. 
Take $n \in \N^*$.
Denote by $R_{x, n}$ the set of $2(n+1)$-jets at $x$ of holomorphic maps $\phi : (U, x) \to \CP^1$ that are ramified at $x$ with ramification degree $n$.
In other words $\phi'(x) = \phi''(x) = \cdots = \phi^{(n)}(x) = 0$ and $\phi^{(n+1)}(x) \neq 0$~:
$$R_{x, n} = \left\{j^{2(n+1)}_x \phi | \text{$\phi$ germ at $x$ of holomorphic $(n+1)$-fold ramified covering with values in $\CP^1$ } \right\}$$
Write also $R^0_{x, n} = \left\{j^{2(n+1)}_x \phi \in R_{x, n} | \phi(x) = 0\right\}$.

Using the identification $(U, x) \simeq (V, 0)$, an element $j^{2(n+1)}_x \phi \in R_{x, n}$ such that $\phi(x) \neq \infty$ can be written~: 
$$j^{2(n+1)}_x \phi = a_0 + a_{n+1} z^{n+1} + a_{n+2}z^{n+2} + \cdots + a_{2(n+1)}z^{2(n+1)}$$
Thus $R_{x, n}$ is identified with $\C \times \C^* \times \C^{n+1}$ by a bijection $\alpha : R_{x, n} \simeq \C \times \C^* \times \C^{n+1}$ defined by $\alpha\left(j^{2(n+1)}_x\phi\right) = \left(a_0, a_{n+1}, \dots, a_{2(n+1)}\right)$. 
Another choice of coodinates on $(U, x)$ induces another identification $\beta : R_{x, n} \simeq \C \times \C^* \times \C^{n+1}$.
The composition $\beta \circ \alpha^{-1}$ is an algebraic automorphism of $\C \times \C^* \times \C^{n+1}$.
As a consequence, $R_{x, n}$ is a smooth algebraic variety of dimension $n+3$. The set $R^0_{x, n} \subset R_{x, n}$ is thus a hypersurface, defined by the equation $(a_0 = 0)$.

The group $G$ acts on the left on the algebraic variety $R_{x, n}$ by postcomposition~: for $g \in G$, $g \cdot j^{2(n+1)}_x \phi = j^{2(n+1)}_x (g \circ \phi)$.
This action is algebraic, and the hypersurface $R^0_{x, n}$ is preserved under the action of the subgroup $H \subset G$.

\begin{proposition}\label{action_G_R}
The group $G$ acts freely on $R_{x, n}$.
Moreover, the equality $G \backslash R_{x, n} = H \backslash R^0_{x, n}$ holds, and this quotient is endowed with a structure of complex manifold of dimension $n$, isomorphic to $\C^n$.
\end{proposition}

\begin{proof}
Since $G$ acts transitively on $\CP^1$, the subset $R^0_{x, n}$ generates $R_{x, n}$ under the action of $G$, i.e. $G \cdot R^0_{x, n} = R_{x, n}$.
Thus $G \backslash R_{x, n} = H \backslash R^0_{x, n}$, and it is enough to show that the action of $H$ on $R^0_{x, n}$ is free.

Take $h \in H$, and let $h = \begin{pmatrix} \alpha & 0 \\ \gamma & \delta \end{pmatrix}$, with $\alpha, \delta \neq 0$. One has
$$
\begin{multlined}
\frac {\alpha\left(a_{n+1}z^{n+1} + \cdots + a_{2(n+1)} z^{2(n+1)} + O\left(z^{2(n+1) + 1}\right)\right)}{\gamma\left(a_{n+1}z^{n+1} + \cdots + a_{2(n+1)} z^{2(n+1)} + O\left(z^{2(n+1) + 1}\right)\right) + \delta} = \frac \alpha\delta \left(a_{n+1} z^{n+1} + \cdots + a_{2n+1}z^{2n+1}\right) + \\ \left( \frac \alpha \delta a_{2(n+1)} - \frac {\alpha \gamma}{\delta^2} a_{n+1}^2 \right)z^{2(n+1)} + O\left(z^{2(n+1)+1}\right)
\end{multlined}
$$
Thus, if $j^{2(n+1)}_x\phi \in R^0_{x, n}$ is given by $j^{2(n+1)}_x\phi = \left(0, a_{n+1}, \dots, a_{2(n+1)}\right)$ in the identification $R_{x, n} \simeq \mathbb C \times \mathbb C^* \times \mathbb C^{n+1}$, one has
$$h \cdot j^{2(n+1)}_x\phi = \left(0, \frac \alpha \delta a_{n+1}, \dots, \frac \alpha \delta a_{2n+1}, \frac \alpha \delta a_{2(n+1)} - \frac {\alpha \gamma}{\delta^2} a_{n+1}^2 \right)$$
thus $h \cdot j^{2(n+1)}_x\phi = j^{2(n+1)}_x\phi$ if and only if $\alpha = \delta$ and $\gamma = 0$, i.e. if and only if $h = \id$.
This shows that the action of $H$ on $R^0_{x, n}$, thus of $G$ on $R_{x, n}$, is free. Moreover the map $\pi : R^0_{x, n} \to \mathbb C^n$ defined by $\pi(0, a_{n+1}, \dots, a_{2(n+1)}) = \left( \frac {a_{n+2}} {a_{n+1}}, \dots, \frac{a_{2n+1}}{a_{n+1}} \right)$ is a holomorphic submersion.
The fibers of $\pi$ are the orbits of $H$.
Although $\pi$ depends on the coordinates on $(U, x)$, if $\pi'$ is the map induced by another coordinate, $\pi' = f \circ \pi$, with $f$ a biholomorphism of $\mathbb C^n$.
Thus $G \backslash R_{x, n} = H \backslash R^0_{x, n}$ is endowed by $\pi$ with the structure of a complex manifold, biholomorphic to $\C^n$.
\end{proof}

\begin{remark}\label{biholomorphisme}
Take $j^{2(n+1)}_x\phi \in R_{x, n}$ and $\alpha$ between two open subsets of $\CP^1$, that is not necessarily a Möbius transformation.
A straightforward computation shows that $j^{2(n+1)}_x(\alpha \circ \phi) \in R_{x, n}$ and $j^{2(n+1)}_x\phi$ are in the same orbit under the action of $G$.
In particular, if $Y$ is a Riemann surface, $\psi : U \to Y$ a branched covering of branching order $n$ at $x$, and $\chi : \psi(U) \to \CP^1$ a holomorphic chart, the orbit of $j^{2(n+1)}_x(\chi \circ \psi) \in R_{x, n}$ under the action of $G$ does not depend on the choice of $\chi$.
As a consequence, the class of $\psi$ in $G \backslash R_{x, n}$ is well-defined.
\end{remark}

\begin{remark}\label{action_Aff_R}
The affine group $\Aff(\C) \subset G$ acts freely on the space
$$\widetilde R_{x, n} = \left\{j^{2n+1}_x \phi | \text{$\phi$ germ at $x$ of branched map of branching order $n$ } \right\}$$
and it follows from the computation in the proof of proposition \ref{action_G_R} that the orbits of $\Aff(\mathbb C)$ in $\widetilde R_{x, n}$ are exactly the images of the orbits of $G$ in $R_{x, n}$ under the obvious projection $R_{x, n} \to \widetilde R_{x, n}$.
As a consequence, $G \backslash R_{x, n} = \Aff(\mathbb C) \backslash \widetilde R_{x, n}$.
\end{remark}

\subsection{Bundles of Branched Projective $2$-Frames}\label{gunning_branché}

Let $D$ be an effective divisor on $X$, and write
\begin{equation}\label{diviseur}
D = \sum_{i = 1}^r n_i \cdot y_i 
\end{equation}
where $n_i \ge 1$ for $1 \le i \le r$, and the $y_i$ are pairwise distinct. Write $\mathbb D = \{y_1, \dots, y_r\}$, and $\mathbb X = X \backslash \mathbb D$.
Denote by $A_X^D = G \backslash R_{y_1, n_1} \times G \backslash R_{y_2, n_2} \times \cdots \times G \backslash R_{y_r, n_r}$.

Let $\alpha \in A_X^D$~: $\alpha = \left(\alpha_1, \dots, \alpha_r\right)$, where for any $i$ between $1$ and $r$, $\alpha_i \in G \backslash R_{y_i, n_i}$.
In other words, as in section \ref{revet_action}, $\alpha_i$ is an orbit for the action of $G$ on the set of jets of order $2(n_i+1)$ of germs at $y_i$ of holomorphic maps to $\CP^1$ with branching order $n_i$.

For any $x \in X$ we write~:
\begin{equation}
\left \{ \begin{aligned} \mathcal P_X^D(\alpha)_x &= \mathcal P_{X, x} &\text{ if $x \in \mathbb X$} \\ \mathcal P_X^D(\alpha)_x &= \alpha_i &\text{ if $x = y_i$, $1 \le i \le r$} \end{aligned} \right.
\end{equation}
So that for any $x \in X$, the group $G$ acts holomorphically, freely and transitively on the manifold $\mathcal P_X^D(\alpha)_x$.

Fix $1 \le i \le r$, and take $U_i$ a neighborhood of $y_i$ in $X$ that does not contain any $y_j$ for $j \neq i$. 
Consider $\phi : U_i \to \CP^1$ a holomorphic map ramified with order $n_i$ at $y_i$, and with no other ramification point.
Suppose $j^{2(n_i+1)}_{y_i}\phi \in \alpha_i$.
For any $x \in U_i \cap \mathbb X$, we set $s(x) = j^2_x\phi$ and $s(y_i) = j^{2(n_i+1)}_{y_i}\phi$, so that for all $x \in U_i$, $s(x) \in \mathcal P_X^D(\alpha)_x$.
There exists a unique structure of holomorphic manifold on $\bigsqcup_{x \in U_i} \mathcal P_X^D(\alpha)_x$ such that $s$ is a section of the holomorphic principal bundle $\bigsqcup_{x \in U_i} \mathcal P_X^D(\alpha)_x \to U_i$.
In restriction to $\mathbb X \cap U_i$, the holomorphic structure defined by $s$ on $\bigsqcup_{x \in U_i} \mathcal P_X^D(\alpha)_x$ coincides with the holomorphic structure of $\mathcal P_X|_{U_i \cap \mathbb X}$.
A straightforward computation based on formula \eqref{osculation} shows that the complex manifold structure on $\bigsqcup_{x \in U_i} \mathcal P_X^D(\alpha)_x$ does not depend on the choice of $\phi$.

As a consequence $\bigsqcup_{x \in X} \mathcal P_X^D(\alpha)_x$, along with its obvious projection on $X$, is endowed with the structure of a holomorphic principal bundle in the neighborhood $U_i$ of each $y_i$.
On the intersections $U_i \cap \mathbb X$, these structures coincide with the one on $\mathcal P_X|_{\mathbb X}$.
We thus have a holomorphic principal bundle $\mathcal P_X^D(\alpha) = \bigsqcup_{x \in U_i} \mathcal P_X^D(\alpha)_x$.

\begin{definition}
The principal bundle $\mathcal P_X^D(\alpha)$ is said to be the \emph{bundle of branched projective $2$-frames} on $X$ with \emph{branching divisor} $D$ and \emph{branching class} $\alpha = (\alpha_1, \dots, \alpha_r) \in A_X^D$.
\end{definition}

As in the unbranched case, the $\PGL(2, \C)$principal bundles $\mathcal P_X^D(\alpha)$ might be more easily imagined as the trivialization bundle of a $\CP^1$-bundle.

\begin{definition}\label{def_CP1_bundle_BPS}
Let $\alpha \in A_X^D$. The $\CP^1$-bundle $G \backslash \left(\CP^1 \times \mathcal P_X^D(\alpha)\right) = H \backslash \mathcal P_X^D(\alpha)$ is called the \emph{branched projective osculating line} on $X$, with branching divisor $D$ and branching class $\alpha$.
It is denoted by $P^D_X(\alpha)$.
\end{definition}

Again, remembering that we chose a preferred point $x_0 = 0 \in \CP^1$, $\mathcal P_X^D(\alpha)$ contains a canonical subbundle, denoted by $\mathcal S_X^D(\alpha)$, that is a principal $H$-bundle, and whose fiber over $x \in X$ is defined as follows~:
\begin{equation}
\left \{ \begin{aligned} \mathcal S_X^D(\alpha)_x &= \mathcal S_{X, x} &\text{ if $x \in \mathbb X$} \\ \mathcal S_X^D(\alpha)_x &= \alpha_i \cap R^0_{y_i, n_i} &\text{ if $x = y_i$, $1 \le i \le r$} \end{aligned} \right.
\end{equation}
In particular the subbundle $\mathcal S^D_X(\alpha)$ defines a canonical section of $P^D_X(\alpha)$, denoted by $s_X^D(\alpha)$.

Pick a function $f$ on an open subset $U \subset X$ whose derivative does not vanish on $\mathbb X \cap U$ and such that $j^{2(n_i+1)}_xf \in \alpha_i$ for any $y_i \in U$.
By definition of $\mathcal P^D_X(\alpha)$, $f$ defines a section $\overline f$ of the bundle $\mathcal P^D_X(\alpha)|_U$, thus a local trivialization of $P^D_X(\alpha)$, that we denote by $(\Pi, \mu) : P^D_X(\alpha) \simeq U \times \CP^1$.
In this trivialization, the canonical section $s_X^D(\alpha)$ is given by $\mu\left(s_X^D(\alpha)(x)\right) = f(x)$.
In particular, the section $s_X^D(\alpha)$ is holomorphic, and so is the subbundle $\mathcal S^D_X(\alpha)$.

As in the unbranched case, the adjoint bundle $\ad\left(\mathcal P_X^D(\alpha)\right)$ is endowed with a filtration $F_1^{X, D}(\alpha) \subset F_2^{X, D}(\alpha) = \ad\left(\mathcal S_X^D(\alpha)\right) \subset \ad(\mathcal P_X^D(\alpha))$, where $F_1^{X, D}(\alpha)$ is a line bundle and $F_2^{X, D}(\alpha)$ is a rank $2$ vector bundle.
For any $x \in X$, $\ad\left(\mathcal P_X^D(\alpha)\right)_x$ is identified with $H^0\left(P^D_X(\alpha)_x, T_{P^D_X(\alpha)_x}\right)$. We then have 

\begin{equation}
\left \{ 
\begin{aligned}
F_2^{X, D}(\alpha)_x &= H^0\left(P^D_X(\alpha)_x, T_{P^D_X(\alpha)_x}\left(-s_X^D(\alpha)(x)\right)\right) = \ad\left(\mathcal S_X^D(\alpha)\right) \\
F_1^{X, D}(\alpha)_x &= H^0\left(P^D_X(\alpha)_x, T_{P^D_X(\alpha)_x}\left(-2s_X^D(\alpha)(x)\right) \right)
\end{aligned}
\right .
\end{equation}

One has $F_2^{X, D}(\alpha) = \ad(\mathcal S_X^D(\alpha)$.
Any $\gamma \in \mathcal P_X^D(\alpha)_x$, identifies $\mathcal P_X^D(\alpha)_x$ with $G$, and thus $\ad(\mathcal P_X^D(\alpha))_x$ with $\mathfrak g$.
If moreover $\gamma \in \mathcal S_X^D(\alpha)_x$, then $F_2^{X, D}(\alpha)_x = \ad(\mathcal S_X^D(\alpha))_x$ is identified to $\mathfrak h$ and $F_1^{X, D}(\alpha)_x$ to $\mathfrak l$.

\subsection{Bundles of Branched Projective $2$-Frames and Pullbacks of the Bundle of Projective $2$-Frames}\label{tirés-arrière}

Let $Y$ be a Riemann surface.
Suppose given a local biholomorphism $f : X \to Y$.
The map $f$ gives rise to an isomorphism of principal $G$-bundles $\delta f : \mathcal P_X \to f^*\mathcal P_Y$, whose inverse $(\delta f)^{-1}$ is defined in the following way.
For all $x \in X$, $(f^* \mathcal P_Y)_x = \mathcal P_{Y, f(x)}$.
Set for any $j^2_{f(x)}\phi \in (f^* \mathcal P_Y)_x$~: $(\delta f)^{-1}_x\left(j^2_{f(x)} \phi\right) = j^2_{x}(\phi \circ f) \in \mathcal P_{X,x}$.
Let $\phi : U \subset Y \to \CP^1$ be a local biholomorphism and let $\overline \phi$ denote the section of $\mathcal P_Y|_U$ induced by $\phi$.
Then $(\delta f)^{-1}(f^*\overline \phi) = \overline {\phi \circ f}$, where $\overline {\phi \circ f}$ denotes the section of $\mathcal P_X|_{f^{-1}(U)}$ induced $\phi \circ f$.
As a consequence, $(\delta f)^{-1}$ is holomorphic.
Moreover $(\delta f)^{-1}$ is clearly equivariant with respect to the action of $G$ on $\mathcal P_X$ and $f^*\mathcal P_Y$, thus $\delta f$ is indeed an isomorphism of principal bundles.
If $f' : Y \to Z$ is a local biholomorphism from $Y$ to a Riemann surface $Z$, one has $\delta(f' \circ f) = f^*(\delta f') \circ (\delta f)$.
If $U \subset X$ is an open subset such that $f|_U$ is a biholomorphism onto its image, then $\delta f|_U$ is given by the identification $\mathcal P_X|_U \simeq \mathcal P_Y|_{f(U)}$ of remark \ref{intrinseque}.

From now on, suppose $f : X \to Y$ is a nonconstant holomorphic map.
As in \eqref{diviseur}, write $D = \sum_{i = 1}^r n_i \cdot y_i $ the branching divisor of $f$, $\mathbb D = \{y_1 \dots, y_r\}$ and $\mathbb X = X \backslash \mathbb D$.
By remark \ref{biholomorphisme}, for all $y_i \in \mathbb D$, $j^{2(n_i+1)}_{y_i}f$ defines a class $\alpha_i$ in $G \backslash R_{y_i, n_i}$.
Write $\alpha = (\alpha_1, \dots \alpha_r) \in A_X^D$.

As in the unbranched case, define an isomorphism of principal bundles $\delta f : \mathcal P_X^D(\alpha) \simeq f^* \mathcal P_Y$ by setting for all $j^2_{f(x)}\phi \in (f^* \mathcal P_Y)_x$~: 
\begin{equation}
\left \{
\begin{aligned}
\mathcal (\delta f)^{-1}\left(j^2_{f(x)}\phi\right) &= j^2_x(\phi \circ f) &\text{ if $x \in \mathbb X$} \\
\mathcal (\delta f)^{-1}\left(j^2_{f(x)}\phi\right) &= j^{2(n_i+1)}_x(\phi \circ f) &\text{ if $x = y_i$, $1 \le i \le r$}
\end{aligned}
\right.
\end{equation}
The same arguments as in the unbranched case show that $f$ is an isomorphism of $G$-principal bundles.
In particular, on the surface $\mathbb X$, the isomorphism $\delta f$ is the one defined in the unbranched case.

As in remark \ref{intrinseque}, one has isomorphisms $\iota_1 : \ad\left(\mathcal S_X^D(\alpha)\right) \simeq f^* \ad(\mathcal S_Y)$, $\iota_2 : \ad\left(\mathcal P_X^D(\alpha)\right)/F_1^{X, D}(\alpha) \simeq f^* \ad(\mathcal P_Y)/F_1^Y$.
Since $\delta f$ maps $\mathcal S_X^D(\alpha)$ to $f^* \mathcal S_Y$, the isomorphism $\iota_1$ (respectively $\iota_2$) maps $F_1^{X, D}(\alpha)$ (resp. $F_2^{X, D}(\alpha)/F_1^{X, D}(\alpha)$) to $f^*F_1^Y$ (resp. $f^* (F_2^Y/F_1^Y)$).
The restriction $f|_{\mathbb X}$ is a local biholomorphism and thus induces  $\iota_3 : J^1(T_X|_{\mathbb X}) \simeq f^*J^1\left(T_Y|_{f(\mathbb X)}\right)$ and $\iota_4 : \widetilde F_2X|_{\mathbb X} \simeq f^*\widetilde F_2^Y|_{f(\mathbb X)}$.
On $\mathbb X$, one gets diagrams analogous to the ones in remarks \eqref{intrinseque_2jet} and \eqref{intrinseque_1jet}.

\begin{remark}\label{ident_restreinte}
\begin{itemize}
\item[(i)] The isomorphism $\iota_1$ restricted to the subbundle $F_1^{X, D}(\alpha)$, composed with the isomorphism $F_1^Y \simeq K_Y$ of proposition \ref{2jet}, gives an isomorphism $I_1 : F_1^{X, D}(\alpha) \simeq f^*K_Y$.
Composing on $\mathbb X$ with the identification $K_{\mathbb X} = K_X|_{\mathbb X} \simeq F_1^{X, D}(\alpha)|_{\mathbb X} = F_1^{\mathbb X}$, one gets an isomorphism $\widetilde I_1 : K_{\mathbb X} \simeq f^* (K_Y|_{f(\mathbb X)})$.
The fact that diagram \eqref{intrinseque_2jet} commutes implies that $\widetilde I_1$ coincides with the dual map of the differential $(df|_{\mathbb X})^{-1}$.

\item[(ii)] Similarly, the isomorphism $\iota_2$, composed with the identification of proposition \ref{1jet}, gives an isomorphism $I_2 : \ad\left(\mathcal P_X^D(\alpha)\right)/\ad\left(\mathcal S_X^D(\alpha)\right) \simeq f^*T_Y$. 
Restricted to $\mathbb X$, $I_2$ induces an isomorphism $\widetilde I_2 : T_{\mathbb X} \simeq f^*\left(T_Y|_{f(\mathbb X)}\right)$.
The commutativity of diagram \eqref{intrinseque_1jet} implies that $\widetilde I_2$ coincides with the differential $df|_{\mathbb X}$.

\item[(iii)] Finally, the isomorphism $\iota_1$ (respectively $\iota_2$) gives after quotienting by $F_1^{X, D}(\alpha)$ (resp. after restricting to $\ad(\mathcal S_X^D(\alpha))$) an isomorphism $\ad(\mathcal S_X^D(\alpha))/F_1^{X, D}(\alpha) \simeq f^*(\ad(\mathcal S_Y)/F_1^Y)$, which, composed with the identifications of propositions \ref{2jet} et \ref{1jet}, gives an isomorphism $I_3 : \ad(\mathcal S_X^D(\alpha))/F_1^{X, D}(\alpha) \simeq f^* \mathcal O_Y = \mathcal O_X$.
Restricted to $\mathbb X$, $I_3$ induces an isomorphism $\widetilde I_3 : \mathcal O_{\mathbb X} \simeq \mathcal O_{\mathbb X}$.
The commutativity of diagram \ref{intrinseque_2jet} (or of diagram \ref{intrinseque_1jet}) implies that $\widetilde I_3$ is identity.
\end{itemize}
\end{remark}

Now let us have a look at the Atiyah bundle $\At\left(\mathcal P_X^D(\alpha)\right)$.
The isomorphism $\delta f : \mathcal P_X^D(\alpha) \simeq f^* \mathcal P_Y$ gives an isomorphism
\begin{equation}
\epsilon : \At\left(\mathcal P_X^D(\alpha)\right) \simeq \At(f^* \mathcal P_Y)
\end{equation}
Moreover, the pullback $f^* \mathcal P_Y$ is endowed with a tautological map $F : f^*\mathcal P_Y \to \mathcal P_Y$, equivariant with respect to the action of $G$, such that the following diagram commutes~:
\begin{equation}
\begin{tikzcd} f^*\mathcal P_Y \arrow[d, twoheadrightarrow] \arrow[r, "F"] & \mathcal P_Y \arrow[d, twoheadrightarrow] \\
X \arrow[r, "f"] & Y \end{tikzcd}
\end{equation}
Since $F$ is equivariant with respect to the action of $G$, the differential $dF$ induces the following morphism, that we also denote by $dF$~:
\begin{equation}
dF : \At(f^* \mathcal P_Y) \to f^*\At(\mathcal P_Y)
\end{equation}
By construction, $dF$ is an isomorphism when restricted to $\mathbb X$.

By composing $\epsilon$ and $dF$, one gets a morphism $dF \circ \epsilon : \At(\mathcal P_X^D(\alpha)) \to f^*\At(\mathcal P_Y)$, that is an isomorphism when restricted to $\mathbb X$, and that induces a morphism of short exact sequences~:
\begin{equation}\label{tiré_arrière_atiyah}
\begin{tikzcd}0 \arrow[r] & \ad\left(\mathcal P_X^D(\alpha)\right) \arrow[d, "\sim" {sloped}, "\iota"{swap}] \arrow[r] & \At\left(\mathcal P_X^D(\alpha)\right) \arrow[r] \arrow[d, "dF \circ \epsilon"] & T_X \arrow[d, "df"] \arrow[r] & 0 \\ 0 \arrow[r] & f^* \ad(\mathcal P_Y) \arrow[r] & f^* \At(\mathcal P_Y) \arrow[r] & f^*T_Y \arrow[r] & 0  \end{tikzcd}
\end{equation}
where the bundle isomorphism $\iota$ is induced by $\delta f$.
Moreover $dF \circ \epsilon$ maps $\At(\mathcal S_X^D(\alpha))$ to $f^* \At(\mathcal S_Y)$.

\subsection{The Structure of Branched Filtered $\SO(3, \C)$-Bundle on $\ad\left(\mathcal P_X^D(\alpha)\right)$}

Take $D$, $\mathbb D$ and $\mathbb X$ as in \eqref{diviseur}.
Let $\alpha \in A_X^D$.

\begin{proposition}\label{tangent_branché}
\begin{itemize}
\item[(i)] The line bundle $F^{X, D}_1(\alpha)$ is canonically isomorphic to $K_X(-D)$.
\item[(ii)]The line bundle $\ad\left(\mathcal P_X^D(\alpha)\right)/\ad\left(\mathcal S_X^D(\alpha)\right)$ is canonically isomorphic to $T_X(D)$.
\item[(iii)] The line bundle $\ad\left(\mathcal S_X^D(\alpha)\right)/F^{X, D}_1(\alpha)$ is canonically isomorphic to $\mathcal O_X$.
\end{itemize}
\end{proposition}

\begin{proof}
The proofs of $(i)$, $(ii)$ and $(iii)$ are essentially the the same, they involve respectively the isomorphisms $I_1$, $I_2$ and $I_3$ of remark \ref{ident_restreinte}. We will only write down the proof of point $(i)$.

One has $\mathcal P_X^D(\alpha)|_{\mathbb X} = \mathcal P_{\mathbb X}$.
Thus, according to corollary \ref{tangent}, one has an isomorphism $\kappa : F_1^{X, D}(\alpha)|_{\mathbb X} \simeq K_X|_{\mathbb X}$.

For any integer $i$ between $1$ and $r$, let $U_i$ be an open neighborhood of $y_i$ in $X$ such that $U_i \cap \mathbb D = \{y_i\}$.
Consider $f_i : U_i \to \CP^1$ a nonconstant map whose branching divisor is $n_i \cdot y_i$ and such that $j^{2(n_i+1)}_{y_i}f_i \in \alpha_i$.
According to section \ref{tirés-arrière}, the map $f_i$ induces an isomorphism of bundles $\delta f_i : \mathcal P_X^D(\alpha)|_{U_i} \simeq f_i^*\mathcal P_{f_i(U_i)}$.
Following remark \ref{ident_restreinte}, $\delta f_i$ provides an isomorphism $I_{1, i} : F_1^{X, D}(\alpha)|_{U_i} \simeq f_i^*K_{f_i(U_i)}$.
Moreover, the dual isomorphism of the differential $df : T_X|_{U_i} \simeq f^*T_{f(U_i)}$ is an isomorphism $df^* : f^*K_{f(U_i)} \simeq K_X|_{U_i}(-n_i \cdot y_i)$.
Thus one has an isomorphism $df^* \circ I_{1, i} : F_1^{X, D}(\alpha)|_{U_i} \simeq K_X|_{U_i}(-n_i \cdot y_i)$.

When restricted to $U_i \cap \mathbb X$, the isomorphisms $\kappa$ and $df^* \circ I_{1, i}$ coincide.
Indeed, according to remark \ref{ident_restreinte}, if $F_1^{X, D}(\alpha)|_{U_i \cap \mathbb X}$ is identified with $K_{U_i \cap \mathbb X}$, then $I_{1, i} \circ \kappa$ coincides $(df^*)^{-1}$.

Thus the $I_{1, i}$ ($1 \le i \le r$) and $\kappa$ can be glued together on the intersections $U_i \cap \mathbb X$, to give an isomorphism $F^{X, D}_1(\alpha) \simeq K_X(-D)$
\end{proof}

A straightforward computation shows that for any $x \in X$, $\ad\left(\mathcal S_X^D(\alpha)\right)$ is the orthogonal of $F_1^{X, D}$ for the killing form on $\ad\left(\mathcal P_X^D(\alpha)\right)_x$.
Thus the vector bundle $\ad\left(\mathcal P_X^D(\alpha)\right)$, along with the filtration $F_1^{X, D}(\alpha) \subset \ad\left(\mathcal S_X^D(\alpha)\right) \subset \ad\left(\mathcal P_X^D(\alpha)\right)$, is a \emph{branched filtered $\SO(3, \C)$-bundle}, as in the following definition, introduced in \cite{biswas-dumitrescu}.

\begin{definition}
A \emph{branched filtered $\SO(3, \C)$-bundle} on $X$ is a rank $3$ holomorphic vector bundle on $X$ endowed with
\begin{itemize}
\item[(i)] A nondegenerate bilinear form $B_x$ on each fiber $W_x$ of $W$, such that $B_x$ varies holomorphically with $x \in X$.
\item[(ii)] An identification $\bigwedge^3 W \simeq \mathcal O_X$ such that the bilinear form induced by $B$ on the fibers of $\bigwedge^3 W$ is the trivial one on $\mathcal O_X$
\item[(iii)] A filtration $F_1 \subset F_2 \subset W$ along with identifications $F_1 \simeq K_X(-D)$ and $F_2 / F_1 \simeq \mathcal O_X$ such that for all $x \in X$, $F_2$ is the orthogonal of $F_1$ for the nondegenerate bilinear form $B_x$.
\end{itemize}
\end{definition}

\subsection{Branched Projective Connections}\label{conn_proj_branchées}

In this subsection we introduce the branched analog of projective connections, that provide examples of \emph{branched Cartan connections}.

\begin{proposition}\label{atiyah_1jet_branché}
There is a morphism~:
\begin{equation}
\Phi_X^D(\alpha) : \At\left(\mathcal S_X^D(\alpha)\right)/F_1^{X, D}(\alpha) \xrightarrow{} \ad\left(\mathcal P_X^D(\alpha)\right)/F_1^{X, D}(\alpha)
\end{equation}
that is an isomorphism when restricted to $\mathbb X$ and such that the following diagram commutes~:
\begin{equation}\label{iso_atiyah_branché}
\begin{tikzcd}
\At\left(\mathcal S_X^D(\alpha)\right)/F_1^{X, D}(\alpha) \arrow[d, twoheadrightarrow] \arrow[r, "\Phi_X^D(\alpha)"] & \ad\left(\mathcal P_X^D(\alpha)\right)/F_1^{X, D}(\alpha) \arrow[d, twoheadrightarrow] \\
T_X \arrow[r, hookrightarrow] & T_X(D) = \ad\left(\mathcal P_X^D(\alpha)\right)/\ad\left(\mathcal S_X^D(\alpha)\right)
\end{tikzcd}
\end{equation}
\end{proposition}

\begin{proof}
Proposition \ref{atiyah_1jet} gives an isomorphism $\Phi_{\mathbb X} : \At\left(\mathcal S_X^D(\alpha)\right)/F_1^{X, D}(\alpha)|_{\mathbb X} \xrightarrow{\sim} \linebreak[4] \ad\left(\mathcal P_X^D(\alpha)\right)/F_1^{X, D}(\alpha)|_{\mathbb X}$.

For each $i$ between $1$ and $r$, let $U_i$ be an open set such that $U_i \cap \mathbb X = y_i$.
Let $f_i : U_i \to \CP^1$ be a holomorphic map whose branching divisor is $n_i \cdot y_i$ and such that $j^{2(n_i+1)}_{y_i} f \in \alpha_i$.
According to proposition \ref{atiyah_1jet}, there exists a canonical isomorphism $\Phi_{f_i(U_i)} : \At\left(\mathcal S_{f_i(U_i)}\right)/F_1^{f_i(U_i)} \xrightarrow{\sim} \ad\left(\mathcal P_{f_i(U_i)}\right)/F_1^{f_i(U_i)}$.
Moreover, the isomorphism $\delta f_i : \mathcal P_X^D(\alpha)|_{U_i} \simeq f_i^* \mathcal P_{f_i(U_i)}$ induces an isomorphism $\ad\left(\mathcal P_X^D(\alpha)\right)/F_1^{X, D}(\alpha)|_{U_i} \simeq f_i^*\ad\left(\mathcal P_{f_i(U_i)}\right)/F_1^{f_i(U_i)}$, as well as the morphism $dF \circ \epsilon : \At\left(\mathcal S_X^D(\alpha)\right)/F_1^{X, D}(\alpha)|_{U_i} \to f_i^*\At\left(\mathcal S_{f_i(U_i)}\right)/F_1^{f_i(U_i)}$ of section \ref{tirés-arrière}.
Thus the pullback $f_i^* \Phi_{f(U_i)}$ induces a morphism $\tau_i : \At\left(\mathcal S_X^D(\alpha)\right)/F_1^{X, D}(\alpha)|_{U_i} \to \ad\left(\mathcal P_X^D(\alpha)\right)/F_1^{X, D}(\alpha)|_{U_i}$.
Since the morphism $dF \circ \epsilon$ is an isomorphism when restricted to $U_i \cap \mathbb X$, so is $\tau_i$.

The isomorphisms $\tau_i$ and $\Phi_{\mathbb X}$ coincide when restricted to $U_i \cap \mathbb X$.
In particular, the $\tau_i$'s and $\Phi_{\mathbb X}$ glue together to give the wanted morphism $\Phi_X^D(\alpha)$.
Diagram \eqref{iso_atiyah_branché} commutes because diagrams \eqref{iso_atiyah} and \eqref{tiré_arrière_atiyah} commute.
\end{proof}

The morphism $\Phi_X^D(\alpha)$ allows us to define the notion of a \emph{branched} projective connection on the bundle $\mathcal P_X^D(\alpha)$, similarly to what we did on the bundle of (unbranched) projective $2$-frames.
Let us first give a definition of a \emph{branched Cartan connection}

\begin{definition}
A \emph{branched Cartan connection} on a holomorphic bundle of $(G, Q)$-spaces $\mathcal Q$ on $X$ along with a reduction to a holomorphic bundle of $(G, Q, x_0)$-spaces $\widetilde Q$ is a morphism $\omega|_{\At\left(\widetilde Q\right)} : \At\left(\widetilde Q\right) \to \ad(\mathcal Q)$ that is an isomorphism in restriction to a Zariski open subset of $X$, and whose restriction to the subbundle $\ad\left(\widetilde Q\right) \subset \At\left(\widetilde Q\right)$ is the identity.
\end{definition}

In our framework, a branched Cartan connection for the bundles $\mathcal S_X^D(\alpha) \subset \mathcal P_X^D(\alpha)$ is a morphism $\omega : \At\left(\mathcal P_X^D(\alpha)\right) \to \ad\left(\mathcal P_X^D(\alpha)\right)$ that is a splitting of the following axact sequence~:
\begin{equation}\label{suite_cartan_branché}
0  \to \ad\left(\mathcal P_X^D(\alpha)\right) \to \At\left(\mathcal P_X^D(\alpha)\right) \to T_X \to 0
\end{equation}
and induces a morphism $\omega|_{\At\left(\mathcal S_X^D(\alpha)\right)} : \At\left(\mathcal S_X^D(\alpha)\right) \to \ad\left(\mathcal P_X^D(\alpha)\right)$ that is an isomorphism when restricted to a Zariski open subset of $X$. 
Such a connection $\omega$ induces the following commutative diagram (using the isomorphism $(ii)$ of proposition \ref{tangent_branché})~:
\begin{equation}\label{diag_connexion_proj_branchée_diviseur}
\begin{tikzcd} 0 \arrow[r] &\ad\left(\mathcal S_X^D(\alpha)\right) \arrow[d, equal] \arrow[r] & \At\left(\mathcal S_X^D(\alpha)\right) \arrow[d, "\omega|_{\At\left(\mathcal S_X^D(\alpha)\right)}"] \arrow[r] & T_X \arrow[d, "\theta"] \arrow[r] & 0 \\ 0 \arrow[r] & \ad\left(\mathcal S_X^D(\alpha)\right) \arrow [r] & \ad\left(\mathcal P_X^D(\alpha)\right) \arrow[r] & T_X(D) \arrow[r]& 0 \end{tikzcd}
\end{equation}
The \emph{branching divisor} of the connection $\omega$ is the vanishing divisor of the morphism $\theta$.

More precisely, a branched Cartan connection for $\mathcal S_X^D(\alpha) \subset \mathcal P_X^D(\alpha)$ gives rise to the following diagram, whose left square is commutative~:

\begin{equation}\label{diag_connexion_proj_branchée}
\begin{tikzcd} 0 \arrow[r] &F_1^{X, D}(\alpha) \arrow[d, equal] \arrow[r] \arrow[dr, phantom, "\circlearrowleft"] & \At\left(\mathcal S_X^D(\alpha)\right) \arrow[d, "\omega|_{\At\left(\mathcal S_X^D(\alpha)\right)}"] \arrow[r] & \At\left(\mathcal S_X^D(\alpha)\right)/F_1^{X, D}(\alpha) \arrow[d, "\Phi_X^D(\alpha)"] \arrow[r] & 0 \\ 0 \arrow[r] & F_1^{X, D}(\alpha) \arrow [r] & \ad\left(\mathcal P_X^D(\alpha)\right) \arrow[r] & \ad\left(\mathcal P_X^D(\alpha)\right) / F_1^{X, D}(\alpha) \arrow[r]& 0 \end{tikzcd}
\end{equation}

\begin{definition}
A \emph{branched projective connection} on $X$, of branching divisor $D$ and branching class $\alpha$ is a branched Cartan connection for $\mathcal S_X^D(\alpha) \subset \mathcal P_X^D(\alpha)$ such that the diagram \eqref{diag_connexion_proj_branchée} commutes.
\end{definition}

Let $\omega$ be a branched projective connexion on $X$, of branching divisor $D$ and branching class $\alpha$.
According to the diagram \eqref{iso_atiyah_branché}, the morphism $\theta$ in diagram \eqref{diag_connexion_proj_branchée_diviseur} is the embedding $T_X \hookrightarrow T_X(D)$.
Thus the branching divisor of $\omega$ as a branched Cartan connection is $D$, and the vocabulary is consistent.

\begin{remark}
The vector bundle $\ad\left(\mathcal P^D_X(\alpha)\right)$, along with the filtration $F_1^{X, D}(\alpha) \subset \ad\left(\mathcal S^D_X(\alpha)\right) \subset \ad\left(\mathcal P_X^D(\alpha)\right)$ and a branched projective connection, is a \emph{branched $\SO(3, \mathbb C)$-oper}, in the sense of \cite{biswas-dumitrescu}.
\end{remark}

\begin{remark}
The morphism $\Phi_X^D(\alpha)$ of proposition \ref{atiyah_1jet_branché}, when restricted to $\mathbb X$, is the isomorphism $\Phi_{\mathbb X}$ of proposition \ref{atiyah_1jet}.
This implies that a branched projective connection on $X$ is a projective connection when restricted to $\mathbb X$.
\end{remark}

The discussion before proposition \ref{structure_affine} can be held \emph{mutatis mutandis} in the case of branched projective connections on $X$, with branching divisor $D$ and branching class $\alpha$.
In particular, the difference between two such connections is given by a section of the vector bundle $\Hom\left(T_X, F^{X, D}_1(\alpha)\right)$.
Thanks to the identification $(i)$ in proposition \ref{tangent_branché}, proposition \ref{structure_affine} becomes in the branched case~:

\begin{proposition}\label{structure_affine_branchée}
The set of branched projective connections on $X$, of branching divisor $D$ and branching class $\alpha$, is either empty, or an affine space directed by the vector bundle $H^0\left(X, K_X^{\otimes 2}(-D)\right)$.
\end{proposition}

\subsection{Branched Projective Structures}\label{struc_proj_branchées}

In this section, we make the link between the classical notion of branched projective structure and the notion of branched projective connection introduced in the previous section.

Let $D$ be a divisor on $X$, like in \eqref{diviseur}.

\begin{definition}
A \emph{branched projective atlas} on $X$ with branching divisor $D$ is the datum of an open covering $(U_i)_{i \in I}$ and nonconstant holomorphic maps $z_i : U_i \to \CP^1$ whose branching divisor is $D|_{U_i}$ and such that for any $i, j \in I$, there exists a Möbius transformation $g_{ij} \in G$ such that $z_i = g_{ij} \circ z_j$ over $U_i \cap U_j$.

Two branched projective atlases of divisor $D$ are said to be equivalent if their union is a branched projective atlas of divisor $D$.
A \emph{branched projective structure} of divisor $D$ is an equivalence class of branched projective atlases of divisor $D$.
\end{definition}

For an overview on the theory of branched projective structures, see \cite{gallo-kapovich-marden} and references therein.

To a branched projective structure on $X$ of divisor $D$ and maximal atlas $(U_i, z_i)_{i \in I}$, one can associate a branching class $\alpha = (\alpha_1, \dots, \alpha_r) \in A_X^D$ by setting $\alpha_k = j^{2(n_k+1)}_{y_k}z_i \mod G$, where $i \in I$ is such that $y_k \in U_i$.

As we saw in section \ref{gunning_branché}, for any $i \in I$, the map $z_i$ defines a section and thus a principal connection for the bundle $\mathcal P_X^D(\alpha)|_{U_i}$.
Since the $z_i$'s differ by composition with an element in $G$, the principal connections they define coincide on the $U_i \cap U_j$'s and thus define a principal connection $\omega$ on $\mathcal P_X^D(\alpha)$.
When restricted to $\mathbb X$, $\omega$ is unbranched and thus makes diagram \eqref{diag_connexion_proj_branchée} commute, as we saw in section \ref{subsection_connexions_projectives_structures_projectives}.
Moreover, if $y_k$ ($1 \le k \le r$) is any branching point of the atlas $(U_i, z_i)$ and $i_0 \in I$ is such that $y_k \in U_{i_0}$, then $\omega|_{U_{i_0}}$ can be seen as the pullback by $z_{i_0}$ of the projective connection on $\CP^1$ corresponding to the trivial projective structure on $\CP^1$.
In the proof of theorem \ref{atiyah_1jet_branché}, the map $\Phi_X^D(\alpha)|_{U_{i_0}}$ was obtained as the pullback of $\Phi_{\CP^1}$.
Moreover, the tautological projective connection on $\CP^1$ makes the diagram \eqref{diag_connexion_proj} commute.
Thus $\omega$ makes the diagram \eqref{diag_connexion_proj_branchée} commute when restricted to $U_{i_0}$.
This shows the following lemma~:

\begin{lemma}\label{struct_branch_connexion}
A branched projective structure induces a branched projective connection, with the same branching divisor and branching class.
\end{lemma}

\subsection{Branched Projective Structures, Branching Classes and Schwarzian Derivative}

Let us recall the following result, that follows from Fuchs's local theory, and that plays a central role in the theory of branched projective structures. See \cite{saint-gervais} for more details.

Let $n \ge 2$ be an integer, $z_0 \in \mathbb C$ and $U \subset \mathbb C$ be a simply connected open neighborhood of $z_0$.

\begin{proposition}\label{fuchs_locale}
There exists a polynomial $P_n \in \mathbb C[X_1, \dots, X_{n+1}]$ such that, for any $\phi = \varphi dz^{\otimes 2}$ holomorphic quadratic differential on $U \backslash \{z_0\}$, (A) and (B) are equivalent~:

\begin{itemize}
\item[(A)] There exists a holomorphic function $f$ defined on $U$ whose branching divisor is exactly $n \cdot z_0$ and such that $\varphi = \mathcal S(f)$ on $U \backslash \{z_0\}$
\item[(B)] The following are all true
\begin{itemize}
\item[(i)] The quadratic differential $\phi$ extends as a meromorphic quadratic differential on $U$ that admits a pole of order $2$ at $z_0$~: $\varphi = \frac{\alpha_{-2}}{(z-z_0)^2} + \frac{\alpha_{-1}}{z-z_0} + \alpha_0 + \alpha_1(z - z_0) + \alpha_2(z-z_0)^2 + \cdots$ \item[(ii)] $\alpha_{-2} = \frac{1-(n+1)^2}2$ \item[(iii)] $P_n(\alpha_{-1}, \alpha_0, \dots, \alpha_{n-1}) = 0$
\end{itemize}
\end{itemize}
Moreover, the polynomial $P_n$ is given by $P_n = \lambda X_{n+1} + \widetilde P(X_1, \dots X_n)$ for some $\lambda \in \mathbb C^*$ and $\widetilde P \in \mathbb C[X_1, \dots, X_n]$.
\end{proposition}

\begin{remark}\label{surj_diffquadmero}
In particular, proposition \ref{fuchs_locale} implies that given $\alpha_{-1}, \alpha_0, \dots, \alpha_{n-2} \in \mathbb C$, there exists a function $f$ defined on $U$ whose branching divisor is $n \cdot z_0$ and such that $\mathcal S(f) = \frac{1-(n+1)^2}{2(z-z_0)^2} + \frac{\alpha_{-1}}{z-z_0} + \alpha_0 + \alpha_1(z-z_0) + \cdots + \alpha_{n-2}(z-z_0)^{n-2} + \mathcal O(z-z_0)^{n-1}$.
Indeed, it is enough to solve the equation $\mathcal S(f) = \frac{1-(n+1)^2}{2(z-z_0)^2} + \frac{\alpha_{-1}}{z-z_0} + \alpha_0 + \alpha_1(z-z_0) + \cdots + \alpha_{n-2}(z-z_0)^{n-2} + \left(-\frac{\widetilde P(\alpha_{-1}, \dots, \alpha_{n-2})} \lambda\right)(z-z_0)^{n-1}$.
\end{remark}

We are now able to prove the converse of \ref{struct_branch_connexion}.
Take a branched projective connection $\omega$, of divisor $D$ and branching class $\alpha$.
The connection $\omega$ is a projective connection when restricted to $\mathbb X$, and thus comes from a projective structure $p$ on $\mathbb X$ according to corollary \ref{coro_equiv_conn_struct}.
For $i$ between $1$ and $r$, let $U_i$ be a neighborhood of $y_i$ and $f_i : U_i \to \CP^1$ a map with branching divisor $n_i \cdot y_i$ and such that $j^{2(n_i+1)}_{y_i} f_i \in \alpha_i$. 
We have seen that $f_i$ defines a branched projective connection on $U_i$ with the same branching class as $\omega$.
In particular, the difference $\omega - \omega_i$ is a quadratic differential $\phi_i \in \Gamma\left(U_i, K_X^{\otimes 2}(-n_i \cdot y_i)\right)$.
Fix an unbranched projective connection $\omega_0$ on $X$. According to corollary \ref{coro_equiv_conn_struct}, $\omega_0$ is given by a projective structure on $X$.
On $\mathbb X$, the difference $\omega - \omega_0$ is a quadratic differential $\psi \in \Gamma(\mathbb X, K_X^{\otimes 2})$ while on $U_i \cap \mathbb X$, the difference $\omega_i - \omega_0$ is a quadratic differential $\psi_i \in \Gamma(U_i \cap \mathbb X, K_X^{\otimes 2})$.
Moreover, on $U_i \cap \mathbb X$, we have $\psi = \phi_i + \psi_i$.
Since $\phi_i$ vanishes up to order $(n_i-1)$ at $y_i$, the Laurent expansions at $y_i$ of $\psi$ and $\psi_i$ in a local chart $z_i$ of the projective structure $\omega_0$ coincide up to order $n_i-1$.
According to proposition $\ref{fuchs_locale}$, there exists a map $h_i$ defined on $U_i$, whose branching divisor is $n_i \cdot y_i$ and such that $\{h_i,z_i\}dz_i^{\otimes 2} = \psi$.
In particular the projective connection induced by $h_i$ on $U_i \cap \mathbb X$ coincides with $\omega$. 
Moreover lemma \ref{bialgébriques} shows that $h_i$ has same branching class $y_i$ as $f_i$, thus $h_i$ defines a branched projective connection with branching class $\alpha_i$, equal to $\omega$ on $U_i$ by analytic continuation.
Thus when restricted to $U_i$, $\omega$ comes from a branched projective structure $b_i$ with branching class $\alpha_i$.
The branched projective structures $b_i$ glue with the unbranched structure $p$ on $U_i \cap \mathbb X$ and thus define a branched projective structure on the whole $X$ whose associated projective connection is $\omega$.

This discussion shows the following proposition~:

\begin{proposition}
The datum of a branched projective structure on $X$ with branching divisor $D$ and branching class $\alpha$ is equivalent to the datum of a branched projective connection on $X$, of divisor $D$ and branching class $\alpha$.
\end{proposition}

\section{Spaces of Bundles of Branched Projective Frames}\label{section_spaces_bpf}

In this section we investigate the analytic structure of the space $A_X^D$ of bundles of branched projective frames for a given curve $X$ with a given divisor $D$.
The associated $\CP^1$-bundles with section, namely the branched projective osculating lines with their canonical sections can be seen as bundles of affine lines, thus there is a map from $A_X^D$ to the moduli space of affine bundles over $X$, that we study.

\subsection{Parametrizations of the Space of Bundles of Branched Projective Frames on a Riemann Surface with Divisors}

Let us work again with the notations of proposition \ref{fuchs_locale}.
Take also a holomorphic function $f$ on $U$, such that $f'$ has a zero of order $n$ at $z_0$.
Fix the following notations~:

$$
\begin{aligned}
f(z) &= a_0 + a_{n+1}(z-z_0)^{n+1} + a_{n+2}(z-z_0)^{n+2} + \cdots + a_{2n+1}(z-z_0)^{n+1} + \mathcal O(z-z_0)^{2n+2}\\
\frac{f''(z)}{f'(z)} &= \frac{n+1}{z-z_0} + \delta_0 + \delta_1(z-z_0) + \cdots + \delta_{n-1}(z-z_0)^{n-1} + \mathcal O(z-z_0)^n\\
\mathcal S(f)(z) &= \frac{1-(n+1)^2}{2(z-z_0)^2} + \frac{\alpha_{-1}}{(z-z_0)} + \cdots + \alpha_{n-1}(z-z_0)^{n-1} + \mathcal O(z-z_0)^n
\end{aligned}
$$

A straightforward computation shows~:

\begin{lemma}\label{bialgébriques}
There are two algebraic automorphisms $D_n : \mathbb C^n \xrightarrow{\sim} \mathbb C^n$ and $S_n : \mathbb C^n \xrightarrow{\sim} \mathbb C^n$ such that for any function $f$, one has $$(\delta_0, \dots, \delta_{n-1}) = D_n \left(\frac{a_{n+2}}{a_{n+1}}, \frac{a_{n+3}}{a_{n+1}}, \dots, \frac{a_{2n+1}}{a_{n+1}}\right)$$ and $$(\alpha_{-1}, \dots, \alpha_{n-2}) = S_n \left(\frac{a_{n+2}}{a_{n+1}}, \frac{a_{n+3}}{a_{n+1}}, \dots, \frac{a_{2n+1}}{a_{n+1}}\right)$$
\end{lemma}

Let $X$ be a Riemann surface.
Recall from section \ref{subsection_affine_structures} that similarly to the case of projective structures, the space of holomorphic affine structures on an open set $U \subset X$ is an affine space directed by the space $\Gamma(U, K_X)$ of holomorphic differentials on $U$.
If $z_1, z_2$ are two coordinates on $U$ with $z_2 = f(z_1)$, the difference between the affine structures given by $z_2$ and $z_1$ is given by $[z_2, z_1]dz_1$, where $[z_2, z_1] = \frac{f''}{f'}$.

Take $D$ an effective divisor on $X$, like in $\eqref{diviseur}$.
Let $\alpha, \alpha' \in A_X^D$ be two branching classes.
Take also, for any $i \in [[1, r]]$, $f_i$ (respectively $g_i$) a holomorphic function defined on a neighborhood $U_i$ of $y_i$, branched to order $n_i$ at $y_i$ (and nowhere else) and whose branching class at $y_i$ is $\alpha_i$ (respectively $\alpha_i'$).
The difference between the two projective (respectively affine) structures defined by $g_i$ and $f_i$ on $U_i \backslash \{x_0\}$ is given by the quadratic differential $\{g_i, f_i\}df_i^{\otimes 2}$ (respectively by the differential $[g_i, f_i]df_i$).

\begin{lemma}
\begin{itemize}
\item[(i)] The quadratic differential $\{g_i, f_i\}df_i^{\otimes 2}$ extends to a section in $\Gamma\left(U_i, K_X^{\otimes 2}(y_i)\right)$. The differential $[g_i, f_i]df_i$ extends to a section in $\Gamma(U_i, K_X)$
\item[(ii)] The $(n_i-1)$-jet of $\{g_i, f_i\}df_i^{\otimes 2}$ at $y_i$ (respectively the $(n_i-1)$-jet of $[g_i, f_i]df_i$) does not depend on the choice of $f_i$ and $g_i$, of respective branching classes $\alpha$ and $\alpha'$.
\end{itemize}
\end{lemma}

\begin{proof}
Let $z_i$ be a local coordinate on $U_i$, centered at $y_i$.
One has $\{g_i, f_i\}df_i^{\otimes 2} = \{g_i, z_i\}dz_i^{\otimes 2} - \{f_i, z_i\}dz_i^{\otimes 2}$.
According to proposition \ref{fuchs_locale}, the coefficient of $\frac 1{z_i^2}$ in $\{f_i, z_i\}$ and $\{g_i, z_i\}$ is $\frac{1-(n_i+1)^2}2$.
Thus $\{f_i, z_i\} - \{g_i, z_i\}$ has a pole of order at most $1$ at $z_i = 0$, and thus $\{g_i, f_i\}df_i^{\otimes 2} \in \Gamma\left(U_i, K_X^{\otimes 2}(y_i)\right)$.
Similarly, $[g_i, f_i]df_i = [g_i, z_i]dz_i - [f_i, z_i]dz_i$ and the residue of both $[g_i, z_i]$ and $[f_i, z_i]$ at $y_i$ is $n$.
Thus $[g_i, z_i] - [f_i, z_i]$ is holomorphic on $U_i$, and $[g_i, f_i]df_i \in \Gamma(U_i, K_X)$.
This shows point $(i)$.

Now if $\widetilde f_i$, $\widetilde g_i$ are two functions of branching class $\alpha_i$ at $y_i$, according to lemma \ref{bialgébriques}, the coefficients of $\{f_i, z_i\}$ and $\{\widetilde f_i, z_i\}$ coincide to order $n_i-2$, and the same is true for $g_i$ and $\widetilde g_i$.
In particular, $\{g_i, f_i\}df_i^{\otimes 2} = \{g_i, z_i\}dz_i^{\otimes 2} - \{f_i, z_i\}dz_i^{\otimes 2}$ coincide with $\{\widetilde g_i, \widetilde f_i\}d\widetilde f_i^{\otimes 2} = \{\widetilde g_i, z_i\}dz_i^{\otimes 2} - \{\widetilde f_i, z_i\} dz_i^{\otimes 2}$ up to order $n_i-2$.
This can also be stated as $j^{n_i-1}_{y_i}\{g_i, f_i\}df_i^{\otimes 2} = j^{n_i-1}_{y_i}\{\widetilde g_i, \widetilde f_i\}d\widetilde f_i^{\otimes 2}$, since $\{g_i, f_i\}df_i^{\otimes 2} \in \Gamma\left(U_i, K_X^{\otimes 2}(y_i)\right)$.
The same argument works \emph{mutatis mutandis} to show $j^{n_i-1}_{y_i}[g_i, f_i]df_i = j^{n_i-1}_{y_i}[\widetilde g_i, \widetilde f_i]d\widetilde f_i$.
This proves $(ii)$.
\end{proof}

Moreover, it follows from remark $\ref{surj_diffquadmero}$ that for any meromorphic quadratic differential around $y_i$ with simple pole at $y_i$, $\phi \in K_X^{\otimes 2}\left([y_i]\right)_{(y_i)}$, there exists a branching class $\alpha_i''$ at $y_i$ such that if $h_i$ has branching class $\alpha_i''$, then $j^{n_i-1}_{y_i} \{h_i, f_i\}df_i^{\otimes 2} = j^{n_i-1}\phi$.
Similarly, for any $j^{n_i-1}_{y_i}\psi \in J^{n_i-1}_{y_i}K_X$, there exists a branching class $\alpha_i'''$ at $y_i$ such that if $h_i$ has branching class $\alpha_i'''$, then $j^{n_i-1}_{y_i}[h_i, f_i]df_i = j^{n_i-1}_{y_i}\psi$.
Indeed a differential form with pole of order $1$ at $y_i$ is given by $[h_i, z_i]dz_i$ for some holomorphic function $h_i$ with branching divisor $n \cdot y_i$ if and only if it has residue $n$ at $y_i$.

Finally, note that $\prod_{i=1}^r J^{n_i-1}_{y_i}K_X^{\otimes 2}\left(D^{\red}\right) = H^0\left(D, K_X^{\otimes 2}\left(D^{\red}\right)|_D\right)$ and $\prod_{i=1}^r J^{n_i-1}_{y_i}K_X = \break H^0(D, K_X|_D)$.

We have just proved the following~:

\begin{proposition}\label{struct_affine_gunning}
The schwarzian derivative equips the algebraic variety $A_X^D$ of branching classes on $(X, D)$ with the structure of an affine space, directed by the vector space $H^0(D, K_X^{\otimes 2}(-D^{\red})|_D)$. 

The differential operator $f \mapsto \frac{f''}{f'}$ equips $A_X^D$ with the structure of an affine space, directed by the vector space $H^0(D, K_X|_D)$.
\end{proposition}

\subsection{Bundles of Branched Projective $2$-Frames as Abstract $\CP^1$-Bundles with Section}\label{subsection_branching_classes_couple_P_s}

Let $D$ be an effective divisor on $X$, as in \eqref{diviseur}.
We have seen that any branching class $\alpha \in A_X^D$ is associated to a $\CP^1$-bundle with a distinguished section, namely the branched projective osculating line $P_X^D(\alpha)$, along with the section $s$ defined in section \ref{gunning_branché}.
In this section, we study the map between $A_X^D$ and the space of isomorphism classes of couples $(P, \sigma)$, where $P$ is a $\CP^1$-bundle on $X$ and $\sigma$ is a section of $P$.
See also \cite{loray-marin} and \cite{mandelbaum_3} for a study of the analytic $\CP^1$-bundle associated to a (branched) projective structure.

The datum of a $\CP^1$-bundle with a section is equivalent to the datum of a $\A^1$-bundle, where $\A^1$ is the complex affine line (with automorphism group the affine group $\Aff(\C)$).
Indeed the section, seen as the section at infinity, provides a reduction of structure group of the $\CP^1$-bundle to the affine group (the group of homographies preserving the point at infinity).
Conversely an $\A^1$-bundle gives rise to a $\CP^1$-bundle along with a section at infinity, by adding a point at infinity to the fibers.
Thus the space of isomorphism classes of $\CP^1$-bundles with section is in fact the space of isomorphism classes of $\A^1$-bundles.

Now an $\A^1$-bundle $A$ on $X$ is an affine bundle directed by a line bundle, namely the line bundle $L$ whose fiber $L_x$ over $x \in X$ is the line of constant vector fields on the fiber $A_x$.
If $A$ is seen as a $\CP^1$-bundle with section $(P, \sigma)$, then $L_x$ is the line of vector fields on the fiber vanishing twice at $\sigma(x)$.
Note that, although there is a canonical action by translation of $L_x$ on $A_x$, multiplication of the vectors of $L_x$ by a nonzero number defines another action.
Thus any automorphism of the line bundle $L$ defines on $A$ another structure of affine bundle directed by $L$.

Now fix $L$ a line bundle on $X$, and let $A$ be an affine bundle directed by $L$.
There is an open cover $(U_i)_{i \in I}$ of $X$ such that $A$ admits a local section $f_i$ on each $U_i$.
Denote by $h_{ij}$ the section of $L$ on $U_i \cap U_j$ defined by $h_{ij} = f_j - f_i$.
The family $(h_{ij})$ is a cocycle that represents a class $c(A) \in H^1(X, L)$.
The class $c(A)$ determines $A$ as an affine bundle directed by $L$.
Now as an $\A^1$-bundle, $A$ has one structure of affine bundle directed by $L$ per automorphism of $L$.
Thus the $\A^1$-bundle $A$ is determined by an orbit of the gauge group $\Aut(L)$ acting on $H^1(X, L)$.
In the case where $X$ is compact, $\Aut(L) = \C^*$ and thus the space of $\A^1$-bundles ($\CP^1$-bundles with section) with underlying line bundle $L$ is $\mathbb P\left(H^1(X, L)\right) \cup \{0\}$.
See \cite{maruyama}, \cite{heu-loray} for a more detailed study of $\CP^1$-bundles on Riemann surfaces.

Let us come back to the case of the branched projective osculating line $P_X^D(\alpha)$ with its section $s$, where $\alpha$ is a branching class on $X$ of divisor $D$.
We suppose moreover that $X$ is compact.
The line bundle whose fiber over $x \in X$ contains the vector fields on $P_X^D(\alpha)_x$ vanishing twice at $s(x)$ is $F_1^{X, D}(\alpha)$, that is canonically identified with $K_X(-D)$ according to proposition \ref{tangent_branché}.
Thus the couple $(P_X^D(\alpha), s)$ is canonically endowed with the structure of an affine bundle directed by $K_X(-D)$.
As a consequence, it is associated to a class $\gamma_{\alpha} \in H^1\left(X, K_X(-D)\right)$, and, as an abstract $A^1$-bundle, to an element $[\gamma_{\alpha}]$ in $\mathbb P\left(H^1\left(X, K_X(-D)\right)\right) \cup \{0\}$ (here the compacity of $X$ matters).
Recall that to compute $\gamma_{\alpha}$, one has to take local sections of $P_X^D(\alpha)$ that do not intersect $s$, consider their differences as local sections of $K_X(-D)$ defining a cocycle, and compute the associated cohomology class.

\begin{lemma}\label{lemme_equivalence_section_affine_structure}
The datum of a local section of $P_X^D(\alpha)$, over an open subset $U \subset X$, that does not intersect $s$ is equivalent to the datum of a branched affine structure on $U$ whose associated branched projective structure has branching class $\alpha$.
\end{lemma}

\begin{proof}
First, suppose given such an affine structure.
Since the branching class of the associated projective structure is $\alpha$, it is given by local trivializations of $P_X^D(\alpha)|_U$, whose changes of trivializations preserve $\infty \in \CP^1$.
Thus it defines a section $\sigma$ of $P_X^D(\alpha)|_U$, defined by the preimage of $\infty$ by the trivializations.
The images of the section $s$ in those trivializations are the local functions from $U$ to $\CP^1$ that are the charts of the affine structure.
Since those charts take values in $\C$, the images of $s$ never take the value $\infty$ and thus $\sigma$ does not intersect $s$.

Conversely, suppose given a section $\sigma$ of $P_X^D(\alpha)$ that does not intersect $s$.
We will show that there exists a unique connection on $P_X^D(\alpha)$ that is a branched projective connection of branching class $\alpha$ and such that $\sigma$ is a flat section.
This will prove the proposition~: among the charts of the branched projective structure associated to this branched projection, consider those whose associated trivilizations of $P_X^D(\alpha)$ send $\sigma$ to $\infty$.
The changes of charts preserve $\infty \in \CP^1$, and thus belong to $\Aff(\C)$.
Moreover, since $\sigma$ does not intersect $s$, the chosen charts take values in $\C$ and thus define an affine structure.

Let us go back to the definition of a branched projective connection on the $G$-principal bundle $\mathcal P_X^D(\alpha)$, given in subection \ref{conn_proj_branchées}.
For any $x \in X$, the fiber $\mathcal P_X^D(\alpha)_x$ can be seen as the set of all isomorphisms $P_X^D(\alpha)_x \xrightarrow{\sim} \CP^1$.
Let us denote by $\mathcal K$ the $\Aff(\C)$-principal subbundle of $\mathcal P_X^D(\alpha)$, whose fiber over $x \in X$ is the set of trivializations $\Psi : P_X^D(\alpha)_x \xrightarrow{\sim} \CP^1$ such that $\Psi(\sigma(x)) = \infty$.
Since $\sigma$ and $s$ do not intersect, it is easy to figure out that $\At(\mathcal K) \oplus F_1^{X, D}(\alpha) = \At\left(\mathcal P_X^D(\alpha)\right)$, $\ad(\mathcal K) \oplus F_1^{X, D}(\alpha) = \ad\left(\mathcal P_X^D(\alpha)\right)$ and $\left(\ad(\mathcal K) \cap \ad\left(\mathcal S_X^D(\alpha)\right)\right) \oplus F_1^{X, D}(\alpha) = \ad\left(\mathcal S_X^D(\alpha)\right)$.
Moreover, the section $\sigma$ is flat for a branched Cartan connection $\omega$ on $\mathcal P_X^D(\alpha)$ if and only if $\omega(\At(\mathcal K)) \subset \ad(\mathcal K)$.
Therefore if $\omega $ is moreover a branched projective connection, then $\omega|_{\At(\mathcal K) \cap \At\left(\mathcal S_X^D(\alpha)\right)}$ is fully determined by the map $\Phi_X^D(\alpha)$ in diagram \eqref{diag_connexion_proj_branchée}.
Since $\omega|_{F_1^{X, D}(\alpha)}$ is also determined (it has to be the identity), $\omega|_{\At\left(\mathcal S_X^D(\alpha)\right)}$, and thus $\omega$, is determined by the datum of $\mathcal K$.
\end{proof}

Let $\sigma_1$, $\sigma_2$ be two local sections of $P_X^D(\alpha)$, over an open subset $U \subset X$, that do not intersect $s$.
Denote by $a_1$ and $a_2$ the associated branched affine structures.
Since $P_X^D(\alpha) \backslash s(X)$ is an affine bundle directed by $K_X(-D)$, the difference $\sigma_2 - \sigma_1$ is a local section of $K_X(-D)$.
Moreover the difference $a_2-a_1$ is a local section of $K_X$ (see \cite{mandelbaum_1}, that extends remarks of section \ref{subsection_affine_structures}).

\begin{lemma}\label{lemme_diff_section_structure}
One has $\sigma_2-\sigma_1 = -(a_2-a_1)$.
\end{lemma}

\begin{proof}
By analytic continuation, it is enough to prove this away from the branched points.
Thus we suppose that $U$ does not contain any point of $D$.
Up to shrinking $U$, suppose given a chart $w : U \to \C$ of the affine stucture $a_1$ on $U$.
The chart $w$ provides a trivialization $(w, z) : P_X^D(\alpha) \xrightarrow{\sim} U \times \CP^1$, such that the section $\sigma_1$ is given by $z \circ \sigma_1 = \infty$, and the section $s$ by $z \circ s = w$.
Let $\lambda : w(U) \to \CP^1$ be the holomorphic section such that $z \circ \sigma_2 = \lambda(w)$.
If $\zeta = 1/(z-w)$, for any $x \in U$, the difference $\sigma_2(x) - \sigma_1(x)$ is given by the vector field on $P_X^D(\alpha)_x$ whose flow at time $1$ sends $\sigma_1(x)$ to $\sigma_2(x)$, namely $\frac{1}{\lambda(w)-w} \partial_{\zeta} = \frac{-2}{\lambda(w)-w} \frac{(z-w)^2}2 \partial_z$.
Since the identification $F_1^{X, D}(\alpha) \simeq K_X(-D)$ of proposition \ref{tangent_branché} identifies $dw$ with  $\frac{(z-w)^2}2 \partial_z$, we have on $U$ : $\sigma_2-\sigma_1 = \frac{-2}{\lambda(w)-w} dw$.

On the other hand, denote by $\omega_1$ and $\omega_2$ the local sections of $K_X \otimes \At(\mathcal P_X^D(\alpha)$ that define the projective connections associated to $a_1$ and $a_2$.
At any $x \in U$ the difference $(\omega_2 - \omega_1)_x(\partial w)$, seen as a vector field of $P_X^D(\alpha)_x$, has to vanish twice at $w(x)$ and has to take the value $\lambda'(w(x))\partial_z$ at $\sigma_2(x)$, because $\sigma_2$ is flat for the connection $\omega_2$.
Thus on the open set $U$, $\omega_2 - \omega_1 = \frac{\lambda'(w)}{(\lambda(w)-w)^2} (z-w)^2 \partial_z dw$.

Now let $\nabla_1$ and $\nabla_2$ be the local morphisms $T_X \to K_X \otimes T_X$ that define the linear connections on $T_X$ induced by the affine structures $a_1$ and $a_2$. 
For $i = 1, 2$, since the projective connection $\omega_i$ preserves the section $\sigma_i$, it preserves the line subbundle $L_i$ of $\ad(\mathcal P_X^D(\alpha))$ whose fiber over $x \in X$ is spanned by $\partial_z$ for $i=1$ and $(z-\lambda(w(x)))^2 \partial_z$ for $i=2$.
Moreover, since $T_s(x)P_X^D(\alpha)$ is canonically identified to $T_{X, x}$ for any $x \in U$, $L_i$ is canonically identified to $T_X|_U$ by $V \in L_{i, x} \mapsto V(s(x)) \in T_s(x)P_X^D(\alpha)$.
The connection $\nabla_i$ is thus given by the restriction to $L_i$ of the parallel transport defined by $\omega_i$. 
Now $\partial_w$ is identified to $\partial_z$ in $L_1$, and is thus a flat section for $\nabla_1$. In $L_2$, $\partial_w$ is identified to $\frac 1 {(w-\lambda(w))^2}(z-\lambda(w))^2\partial_z$.
Thus one has, using this latter identification,
$$
\begin{aligned}
&(\nabla_2 - \nabla_1)(\partial_w)(\partial_w) = \nabla_2(\partial_w)(\partial_w)\\
&= \frac d {dw} \left( \frac 1 {(w-\lambda(w))^2}(z-\lambda(w))^2\partial_z \right) + \left[\frac{\lambda'(w)}{(\lambda(w)-w)^2} (z-w)^2 \partial_z, \frac 1 {(w-\lambda(w))^2}(z-\lambda(w))^2\partial_z \right]\\
&= \left( \frac{-2(1-\lambda'(w))}{(w-\lambda(w))^3}(z-\lambda(w))^2 + \frac{-2\lambda'(w)(z-\lambda(w))}{(w-\lambda(w))^2} - \frac {2\lambda'(w)}{(w-\lambda(w))^3}(z - \lambda(w))(z-w) \right) \partial_z\\
&=\frac{-2}{(w-\lambda(w))^3}(z-\lambda(w))^2\partial_z\\
&=\frac{-2}{w-\lambda(w)}\partial_w
\end{aligned}
$$

Finally, $a_2-a_1 = \nabla_2-\nabla_1 = \frac{2}{\lambda(w)-w} dw = -(\sigma_2-\sigma_1)$

\end{proof}

Fix a branching class $\alpha_0 \in A_X^D$, and let $(\sigma_i^0)_{i \in I}$ be a family of sections of $P_X^D(\alpha_0)$ adapted to an open covering of $X$ $(U_i)_{i \in I}$ and that do not intersect $s$.
We have seen that the class $\gamma_{\alpha_0}$ of the cocycle $(\sigma_j^0-\sigma_i^0)_{i, j \in I}$ in $H^1(X, K_X(-D))$ characterizes the $\CP^1$-bundle $P_X^D(\alpha_0)$, along with its canonical section. Now let $\alpha_1 \in A_X^D$ be another branching class, and $(\sigma_i^1)_{i \in I}$ a family of local sections of $P_X^D(\alpha_1)$ that does not intersect the canonical section, so that $\left(\sigma_j^1 - \sigma_i^1\right)_{i, j \in I}$ represents the class $\gamma_{\alpha_1} \in H^1\left(X, K_X(-D)\right)$.
Now the difference $\gamma_{\alpha_1} - \gamma_{\alpha_0}$ is represented by the cocycle $\left(\left(\sigma_j^1- \sigma_j^0\right) - \left(\sigma_i^1 - \sigma_i^0\right)\right)_{i, j \in I}$.
Thus lemma \ref{lemme_diff_section_structure} implies that $-(\gamma_{\alpha_1} - \gamma_{\alpha_0})$ is represented by the cocycle $\left(\left(a_j^1- a_j^0\right) - \left(a_i^1 - a_i^0\right)\right)_{i, j \in I}$, where $a_i^k$ is the affine structure induced by $\sigma_i^k$ ($k = 1, 2$).

But by proposition \ref{struct_affine_gunning}, the difference $\alpha_1-\alpha_0$ corresponds to an element of $H^0(X, K_X|_D)$, and the correspondance is given by the differential operator $f \mapsto \frac{f''}{f'}$, applied to the charts of $\alpha_1$ written in the charts of $\alpha_0$.
It is equivalent to say that $\alpha_1-\alpha_0$ is given by the differences $a_i^1-a_i^0$ restricted to the points of $D$.
As a consequence, $\gamma_{\alpha_1} - \gamma_{\alpha_0}$ is given by $-\delta(\alpha_1-\alpha_0)$, where $\delta : H^0(X, K_X|_D) \to H^1(X, K_X(-D))$ is the morphism in cohomology given by the short exact sequence $0 \to K_X(-D) \to K_X \to K_X|_D \to 0$. We have proven~:

\begin{theorem}
Let us endow $A_X^D$ with its structure of affine space directed by the vector space $H^0(X, K_X|_D)$.
Let $\gamma : A_X^D \to H^1(X, K_X(-D))$ be the map that maps a branching class $\alpha$ to the isomorphism class of $P_X^D(\alpha)$ along with its canonical section, in the space of affine bundles directed by $K_X(-D)$. 
Then $\gamma$ is a morphism of affine spaces, directed by the linear map $-\delta : H^0(X, K_X|_D) \to H^1(X, K_X(-D))$, where $\delta$ is the map induced in cohomology by the short exact sequence $0 \to K_X(-D) \to K_X \to K_X|_D \to 0$.
\end{theorem}

\begingroup
\sloppy
\printbibliography
\endgroup

\end{document}